\documentclass{article}
\usepackage{fullpage}

\usepackage{amsthm}
\usepackage{amsmath}
\usepackage{amssymb}
\providecommand{\abs}[1]{\left|#1\right|}
\providecommand{\norm}[1]{\lVert#1\rVert}
\providecommand{\re}[1]{\text{Re}\left(#1\right)}
\providecommand{\im}[1]{\text{Im}\left(#1\right)}
\usepackage{graphicx}
\usepackage{epstopdf}

\usepackage{caption}
\usepackage{subcaption}

\graphicspath{{images/}}
\usepackage{authblk}

\usepackage{mathtools} 
\newtheorem*{remark*}{Remark}

\newtheorem*{definition*}{Definition}

\newtheorem*{theorem*}{Theorem}
\newtheorem{theorem}{Theorem}

\newcounter{saveeqn}

\providecommand{\keywords}[1]{\textbf{\textit{Keywords:}} #1}

\begin{document}
\pagestyle{plain}
\title{The Numerical Unified Transform Method for Initial-boundary Value Problems on the Half-line}

\author[$\dagger$]{Bernard Deconinck}
\author[$\ddag$]{Thomas Trogdon}
\author[${\S}$]{Xin Yang \thanks{Corresponding author.}}
\affil[$$]{Department of Applied Mathematics, University of Washington, Seattle, WA 98195}
\affil[$$]{$^{\dag}$deconinc@uw.edu, $^{\ddag}$trogdon@uw.edu, $^{\S}$yangxin@uw.edu}
\maketitle

\begin{abstract}
We implement the Unified Transform Method of Fokas as a numerical method to solve linear partial differential equations on the half-line. The method computes the solution at any $x$ and $t$ without spatial discretization or time stepping.  With the help of contour deformations and oscillatory integration techniques, the method's complexity does not increase for large $x,t$ and the method is more accurate as $x,t$ increase. Our goal is to make no assumptions on the functional form of the initial or boundary functions while maintaining high accuracy in a large region of the $(x,t)$ plane. 
\end{abstract}

\section{Introduction}
Standard methods for solving linear partial differential equations (PDEs), including separation of variables and classical integral transforms, are often limited by the order of the PDE and the type of boundary conditions. The unified transform method (UTM), also known as the Method of Fokas~\cite{fokas1997}, is a relatively new method for analyzing a large family of PDEs with general initial and boundary conditions~\cite{fokas2002a}.  When applied to initial boundary value problems (IBVPs) for linear, constant coefficient PDEs, the UTM provides the solutions in terms of contour integrals involving the given initial and boundary conditions~\cite{deconinck2014}. This does not only give rise to new analysis but it also provides a new direction for numerical methods. With this integral representation of the solution, it is possible to compute the solution at any $x,t$ directly. The numerical unified transform method (NUTM) is a numerical method built upon the solution formula from the UTM with the addition of systematic contour deformations. In stark contrast to classical numerical PDE methods such as finite-difference methods, spectral methods and finite-element methods, the NUTM can solve equations in unbounded domains and it does not experience  accumulation of errors or stability issues. These issues that appear in standard numerical methods for evolutionary PDEs do not appear in the NUTM because spatial discretization and time stepping are not required. 

Since the first paper on the NUTM in 2008 \cite{flyer2008}, the method has been applied to the heat equation $q_t=q_{xx}$ on the half-line \cite{flyer2008,fokas2009} and on finite intervals \cite{papa2009}, to the Stokes equations $q_t\pm q_{xxx}=0$ on the half-line \cite{flyer2008} and on finite intervals \cite{kesici2018}, and the advection-diffusion equation $q_t+q_x=q_{xx}$ on the half-line \cite{barros2019}. These applications of the NUTM use fixed contours that do not depend on $(x,t)$ and rely on knowing closed-form expressions for the transforms of initial and boundary data. We refer to such implementations of the NUTM as fixed contour methods (FCMs). As we will see in Section \ref{exheat}, FCMs become less accurate for large $x,t$. 
In contrast to those FCMs, we propose a new implementation of the NUTM that uses contours depending dynamically on $x,t$ and that does not severely restrict the initial or boundary conditions. Our goal is to make no assumptions on the functional form of the initial or boundary functions, other than to restrict them to be in certain function spaces ({\it i.e.}, impose specific decay). We maintain high accuracy in a large region of the $(x,t)$ plane. To summarize, we build up the NUTM to include the following features:
\begin{enumerate}
\item The assumptions on the initial and boundary conditions are significantly weakened compared to the FCM. Decay and regularity conditions are necessary for the purpose of achieving high accuracy. We emphasize that {\bf closed-form expressions for the transforms of initial or boundary conditions are not required}.
\item The method is {\bf uniformly accurate} in that the computational cost to compute the solution at a point $(x,t)$ with given accuracy remains bounded for large $x,t$.
\item The method is {\bf spectrally accurate} in that the error at fixed $(x,t)$, $E_{\text{NUTM}}(N,x,t)=\mathcal{O}(1/N^l)$  for any integer $l$, where $N$ is the number of function evaluations. For certain equations such as the heat equation, it is possible to achieve spectral accuracy uniformly as long as $(x,t)$ are bounded away from $x=0$ and $t=0$.  
\end{enumerate}
These features exist in the numerical inverse scattering transform (NIST) we have implemented for nonlinear integrable PDEs on the whole line \cite{trogdon2012a,trogdon2012b,yang2019}. 
Having studied the solution of the IVP of nonlinear integrable equations and the solution of the IBVP of linear constant coefficient equations, we are set up to understand the numerical issues associated with IBVPs for nonlinear integrable PDEs~\cite{fokas2002b,fokas2005}. Ultimately, we wish to compute the solution of the IBVP of nonlinear integrable equations using the NUTM in a similar fashion.

The paper is organized as follows: Section \ref{sec_prelim} gives a brief overview of the UTM and the methods for oscillatory integrals that are required in what follows. In Section \ref{ch_heat} we discuss the NUTM for the heat equation where the deformation is based on the method of steepest descent. In Section \ref{ch_lnls} we discuss the NUTM applied to the linear Schr\"odinger equation where methods other than the method of steepest descent are needed. In Section~\ref{ch_lkdv}, we show how to apply the NUTM to a third-order PDE with an advection term giving rise to integrands with branch points. Numerical examples are provided throughout. In many examples the initial and boundary conditions are chosen to have closed-form transforms for the purpose of computing the true solution for comparison. An example with the boundary condition which does not have a known expression for the transform is shown at the end of Section 5.2. The proof of the uniform convergence of the NUTM applied to the heat equation is given in the Appendix.

%
\section{Preliminaries}
\label{sec_prelim}
\subsection{The unified transform method on the half-line}
 Consider a linear PDE written as
\begin{align}
q_t+\omega(-i\partial_x) q =0,
\label{linearpde}
\end{align}
for $x,t>0$. We assume $\omega(k)$ to be a polynomial of degree $p$. Note that $q(x,t)=e^{ikx-\omega(k)t}$ satisfies (\ref{linearpde}). This definition of the dispersion relation $\omega$ typically used in the UTM differs from the common convention by a factor of $i$. The UTM solves IBVPs for (\ref{linearpde}) using transforms of the initial and boundary values,
\begin{align}
\hat{q}_0(k) &= \int_0^{\infty}e^{-ikx}q_0(x,0)dx,\\
\label{qhat}
\tilde{g}_0(\omega(k),t)&=\int_0^t e^{\omega(k) s}q(0,s)ds,\\
\label{gtilde}
\vdots \nonumber \\
\tilde{g}_{p-1}(\omega(k),t)&=\int_0^t e^{\omega(k) s}\frac{\partial^{p-1} q}{\partial x^{p-1}}(0,s)ds.
\end{align}
The number of boundary conditions required for a well-posed problem is determined by the UTM. It is based on the order of  the highest spatial derivative as well as the leading coefficient of $\omega$ \cite{flyer2008}. The solution formula from the UTM depends on contour integrals of the type
\begin{align*}
I_m&=\int_{\mathcal{C}^J_m}e^{ikx-w(k)t}\hat{q}_0(\nu_m(k))dk, \quad  m=1,2,\ldots,p,  \\
B_m&=\int_{\mathcal{C}^I_m}e^{ikx-w(k)t}f_m(k) \hat{g}_m(\omega(k),t)dk, \quad  m=0,1,\ldots, n,
\end{align*}
where $p$ is the degree of $\omega(k)$ and $\nu_m(k)$ is its $m$th symmetry\footnote{A symmetry $\nu(k)$ of $\omega(k)$ satisfies $\omega(\nu(k))=\omega(k)$. The symmetries play an important role in the UTM. 
The $n$ symmetries  $\{\nu_m(k):m=1,2,\dots,n \}$ exist by the fundamental theorem of algebra, and can be chosen to be analytic outside a compact set \cite{utmbook}.}, and $f_m(k)$ is a function explicitly determined by $\omega(k)$, independent of the initial and boundary data.  Thus, the solution to (\ref{linearpde}) can be computed by quadrature. However, the integrands on the contours $\mathcal{C}^I_m$ and $ \mathcal{C}^B_m$ obtained by the UTM are often highly oscillatory, and suitable methods must be applied for an accurate solution.

\subsection{Methods for oscillatory integrals}
The exponential factor $e^{ikx-w(k)t}$ in the integrand is the main cause of oscillations. Deformations based on the method of steepest descent \cite{millerbook} change the oscillations into exponential decay. Define the phase function $\theta(k;x,t)=ikx-w(k)t$. Saddle points $k_0$ satisfy
\[
\frac{d\theta(k;x,t)}{dk}\Big|_{k=k_0}=0.
\]
Near $k=k_0$,
\[
\theta(k;x,t)=ik_0x-w(k_0)t-\frac{w''(k_0)t}{2}(k-k_0)^2 +\mathcal{O}(k-k_0)^3.
\]
The integrand is (locally) exponentially decaying if $k$ follows a path such that $-w''(k_0)t(k-k_0)^2/2$ is negative and decreasing.  
Since the integrals along the deformed paths are exponentially localized near the saddle point, they can be computed with high accuracy with standard quadrature methods after appropriate truncation. For improved accuracy, Gauss-Hermite or Gauss-Laguerre quadratures are suitable, depending on the form of the exponentials and the paths \cite{gibbs2019,huybrechs,uspensky1928}. We choose Clenshaw-Curtis quadrature for the deformed contour integrals for convenience, as it is spectrally accurate and efficient in most cases \cite{trefethen2008}. We note that there are situations where the deformations are restricted and the method of steepest descent is not applicable, see Sections \ref{ch_lnls} and \ref{ch_lkdv}.     

The region in the complex $k-$plane where the contour can be deformed depends on the analyticity of the transform data $\hat{q}_0(k)$ and $\hat{g}_m(\omega(k),t)$ which is related to the decay rate of the initial and boundary data. For instance, when $q(x,0)$ and $q(0,t)$ are integrable, $\hat{q}_0$ is analytic and bounded in the lower-half plane $\{k\in \mathbb{C}:\im{k} < 0\}$ and $\hat{g}_0(\omega(k),t)$ is analytic and bounded in $\{k\in \mathbb{C}:\re{\omega(k)} < 0\}$. Data with faster decay gives more freedom to deform the contour. We consider data with exponential decay rate $\delta>0$, defined by
 \[
 C_\delta^m=\left\{f \in C^{m}([0,\infty)),~  \exists \delta'>\delta,~\text{such that}~\sup_{x \in [0,\infty)} e^{\delta' x} \abs{f(x)} <\infty  \right\}.
 \]
\paragraph{Remark.} For $f\in  C_\delta^m$ , we have $ \int_{0}^{\infty} e^{\delta'' x}\abs{f(x)}dx <\infty$ with $\delta''=\frac{\delta'+\delta}{2}>\delta$. The boundedness is introduced for convenience in the proofs in the appendix and the implied integrability is used to allow deformation of contours not just in the interior of regions but also to their boundaries.\\

If the initial condition $q_0 \in  C_\delta^m $, then $\hat{q}_0$ is analytic and bounded in a open set containing $\{k\in \mathbb{C}:\im{k} \leq \delta\}$. Therefore contour integrals of $\hat{q}_0(k)$ can be deformed inside a larger region.
When the contours get close to the boundary of regions in which they can be deformed, highly oscillatory integrals of the form
\begin{align}
S(x,t)=\int_{k_0}^{\infty} f(k) e^{\theta(k;x,t)}dk,
\label{oscint}
\end{align}
appear. 
Here $f(k)$ is, in general, not analytically extendable off the real axis, $k=k_0$ is the critical point of $\theta(k;x,t)$ and $\omega(k)\in i\mathbb{R}$. This integral is highly oscillatory when the parameters $x,t$ are large, and therefore with traditional numerical quadrature methods the cost to achieve a desired accuracy increases as $x,t$ increase. Fortunately, there are methods specific to highly oscillatory integrals, such as Filon-type and Levin-type methods, that are more accurate as oscillations increase, with a fixed number of evaluations of the integrand \cite{iserles2006}. Hence it is still possible to attain uniform accuracy without an increasing computational cost. On the other hand, unlike in the method of steepest descent, the global error over all $x,t$ does not, in general, decay spectrally. While we do compute solutions at arbitrarily large $x,t$ with increasing accuracy as $x,t$ increase, improvements over our methodology in the computation of integrals of the type given in~(\ref{oscint}) will improve the overall efficiency of our method. Some possible directions for the improved evaluation of~(\ref{oscint}) are: 
\begin{enumerate}
\item Better computational methods for oscillatory integrals that can achieve higher order of accuracy, and 
\item Faster solvers that can handle more nodes/modes like the ultraspherical polynomial spectral method~\cite{olver2012}. 
\end{enumerate}
We emphasize that our work focuses on the integrals from the UTM and therefore we focus on analyticity and decay of the integrands and possible contour deformations. A complete discussion of the treatment of~(\ref{oscint}) is beyond the scope of this paper as any improvement is not only relevant to the NUTM but is also worth studying for its own sake. 

In order to make use of the path of steepest descent to obtain exponential localization, we avoid computing the solution with arbitrarily small $x$ or $t$. Hence in discussion about uniform accuracy, we assume $x,t\geq c$ for some constant $c>0$. We choose $c=0.1$ in most examples for convenience.
\paragraph{Remark.} The NUTM is less efficient for small $x$ or $t$. We can use  extrapolation and Taylor expansions to get $q(x,t)$ with small $x$ or $t$ \cite{trogdon2019}. Traditional time-stepping methods can be powerful and convenient if the number of time steps is small.\\ 

Methods for oscillatory integrals are also needed for computing the transforms $\hat{q}, \tilde{g}$.
These transformed data are Fourier-type integrals that can be handled efficiently by Levin's method. In Figure \ref{heatq0levin},  the absolute errors for $\hat{q}_0(x+i)$ for $q_0(x)=e^{-2x}$ are plotted. The number of collocation points $N=40$ is the same for Levin's method and for Clenshaw-Curtis quadrature. The values start to diverge for large $x$ for Clenshaw-Curtis quadrature when the oscillations are under resolved but Levin's method provides reliable approximations with decreasing errors.
\begin{figure}
  \makebox[\textwidth][c]{
  \includegraphics[width=0.5\textwidth]{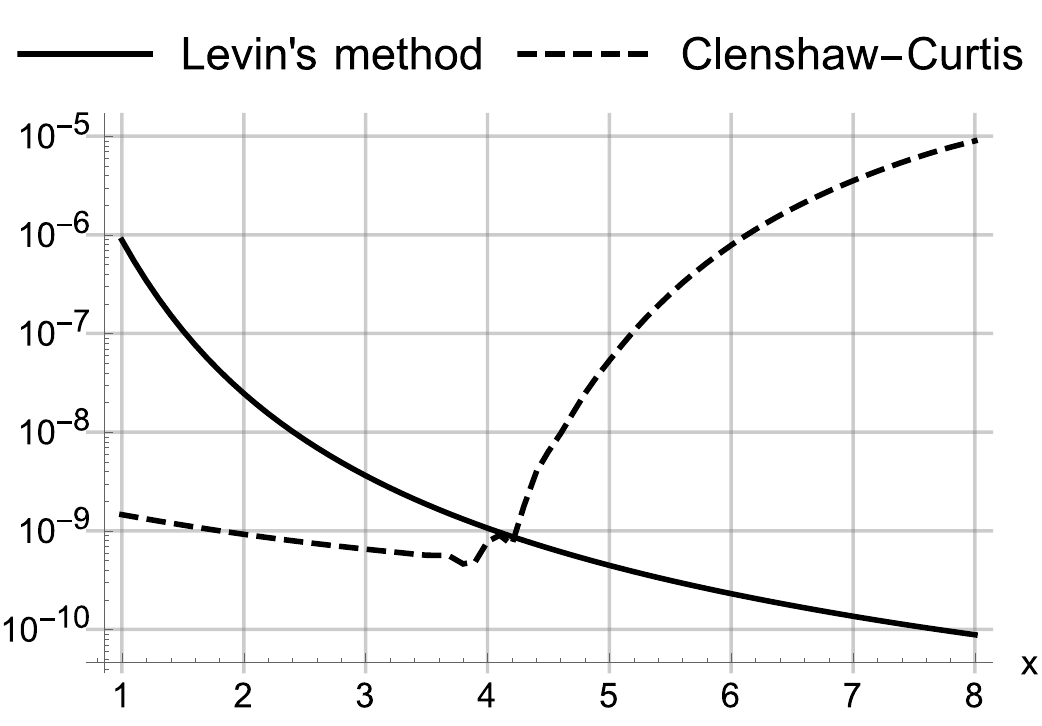}
  }
  \caption{The absolute errors for the computation of $\hat{q}_0(x+i)$ for $x\in [1,8]$. The curves are computed using  Clenshaw-Curtis quadrature (dashed) and Levin's method (solid). Both methods use a fixed number of nodes $N=40$. The initial data is $q_0(x)=e^{-2x}$.}
  \label{heatq0levin}
\end{figure}

\section{The heat equation on the half-line}
\label{ch_heat}
We consider the heat equation on the half-line,
\begin{align}
q_t=q_{xx}, ~~~t>0,\,~~ x>0,
\label{heateq}
\end{align}
with Dirichlet boundary data $q(0,t)=g_0(t)$ and initial data $q(x,0)=q_0(x)$. The dispersion relation for the heat equation is $\omega(k)=k^2$. The initial data $q_0$ is assumed to be in $ C_\delta^{\infty}$ for some  $\delta>0$ and the boundary data $g_0$ is assumed to be in $ C_\gamma^{\infty}$ for some  $\gamma>0$.  The smoothness of $q_0,~g_0$ allows us to compute the transformed data $\hat{q},~\tilde{g}$ accurately. The rate of decay affects the regions where the deformation of the integration path is allowed. The same methodology can still be applied, with less efficiency and accuracy, when weaker conditions are satisfied. 
\paragraph{Remark.}
It is possible to deal with non-decaying boundary data when the asymptotics of the data is known and can be handled by some other method. The UTM for linear PDEs with piecewise-constant data is studied in \cite{trogdon2019}. Since the equation is linear, if the data is given as a superposition of data, it may then be beneficial to obtain the solution of the full problem as a superposition of solutions corresponding to individual pieces of data. For instance, suppose $g_0(t)=h_1 + h_2(t)$ where  $h_1$ is a constant and $h_2 \in C_\delta^{\infty}$. The transform $\hat{h}_1(k,\infty)=-1/k^2$ is a meromorphic function in $\mathbb{C}$ and there is no restriction about where the integral contour for $\hat{h}_1(k,\infty)$ can be deformed if the residue is collected correctly. The full solution is easily obtained by superimposing the NUTM solutions for the problems corresponding to $h_1$ and $h_2$ separately. 

 \subsection{The solution formula from the unified transform method}
The solution to the heat equation on the half-line with Dirichlet boundary condition is 
\begin{align}
q(x,t)=\frac{1}{2\pi}\int_{-\infty}^{\infty} e^{ikx-\omega(k)t}\hat{q}_0(k)dk-\frac{1}{2\pi}\int_{\partial D^+}e^{ikx-\omega(k)t}\left[\hat{q}_0(-k)+2ik\tilde{g}_0(\omega(k),t)\right]dk,
\label{solheat2}
\end{align}
where the contour $\partial D^+=\{re^{i\pi/4}:r\in [0,\infty)\}\cup \{re^{3i\pi/4}:r\in [0,\infty)\}$ is the boundary of the region $D^+=\{(re^{iu}):r\in (0,\infty),u\in (\pi/4,3\pi/4)\}$, shown in Figure \ref{heatcpa} \cite{utmbook}. The transformed data $\hat{q}_0(k)$ and $\tilde{g}_0(\omega(k),t)$ are defined by (\ref{qhat}) and (\ref{gtilde}) respectively.
\begin{figure}
     \centering
     \begin{subfigure}{0.45\textwidth}     
     \includegraphics[width=1. \textwidth]{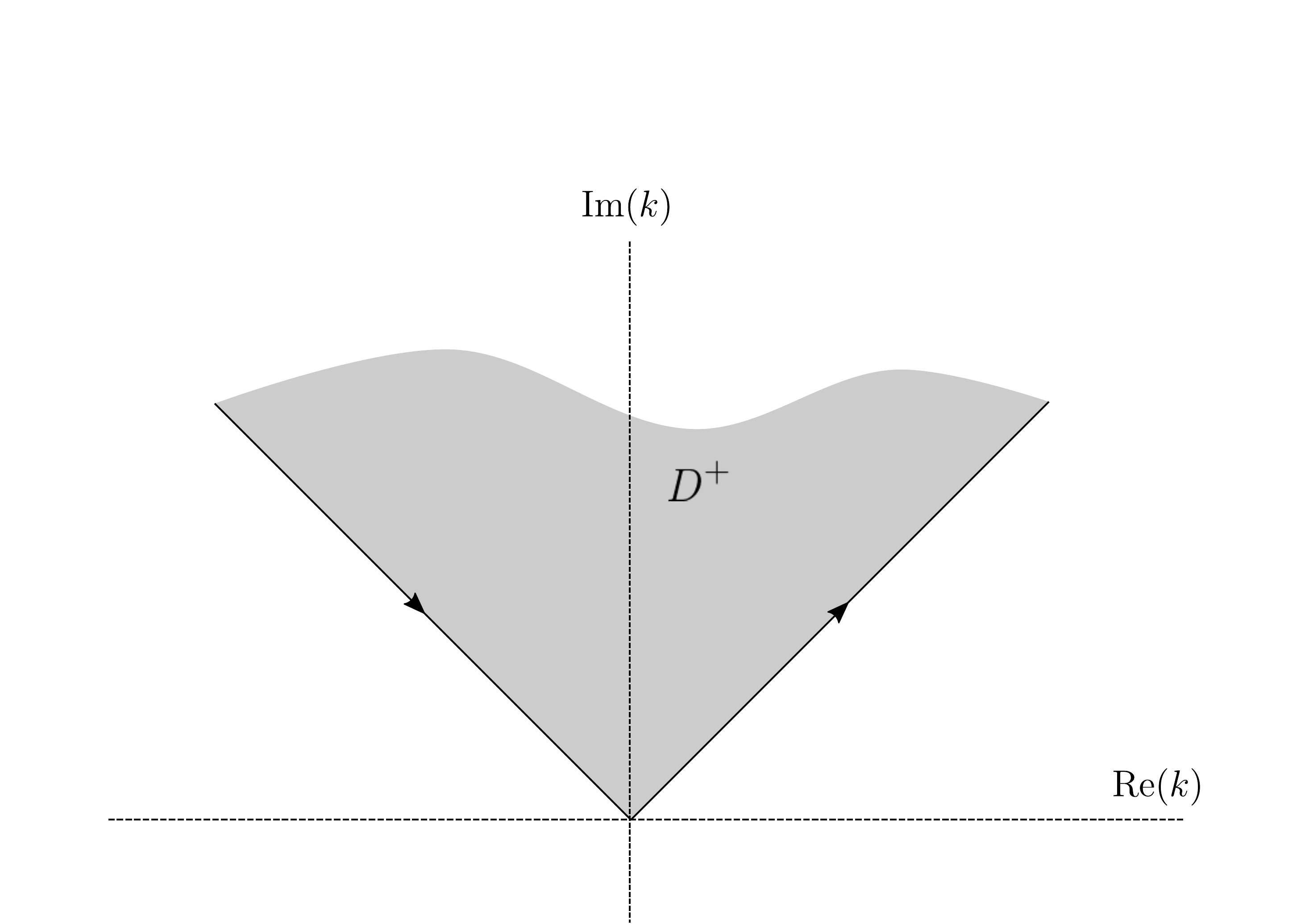}
     \caption{}
     \end{subfigure}
     \begin{subfigure}{0.45\textwidth}     
     \includegraphics[width=1. \textwidth]{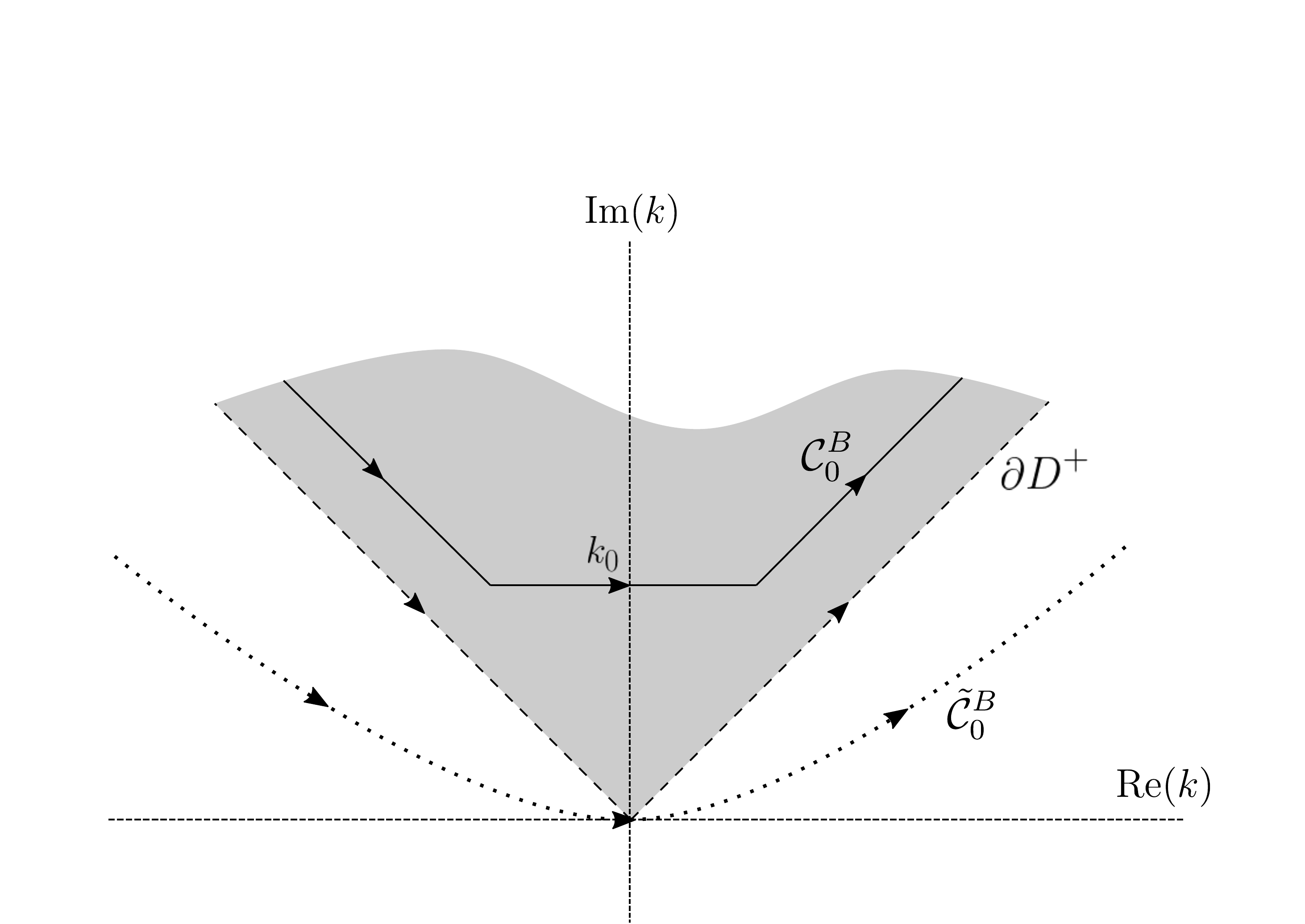}
     \caption{}
     \end{subfigure}
      \caption{Regions for the heat equation. Panel (a) shows the region $D^+=\{\re{k^2}<0\} \cap \mathbb{C}^+$. Panel (b) shows different integral paths for $B_0$: (i) $\partial D^+$: the undeformed contour  (dashed),  (ii) $\mathcal{C}_0^B$: the deformed contour across the saddle point $k_0$ (solid), and (iii) $\tilde{\mathcal{C}}_0^B$: the deformed contour used in \cite{flyer2008} (dotted).}
      \label{heatcpa}
 \end{figure}

\begin{figure}
     \centering
      \includegraphics[width=0.7\textwidth]{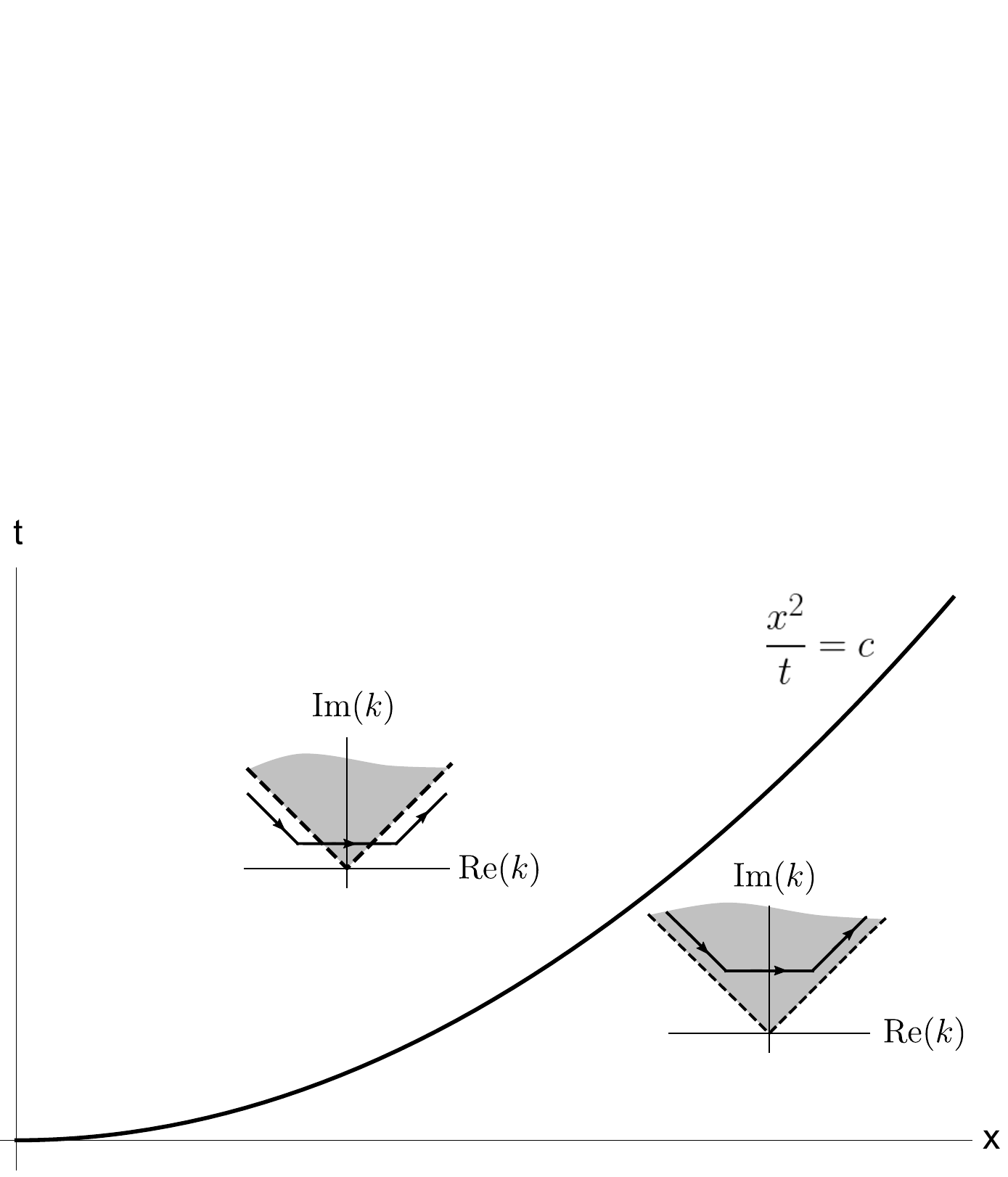}
	  \caption{Deformed contours for the heat equation. Depending on the values of $(x,t)$, the deformed contour for $B_0$ can be inside or outside of $ \partial D^+$. Solid lines represent the deformed contours. Dashed lines give $\partial D^+$, where $D^+$ is the shaded region. }
 	 \label{heatcpb}
\end{figure}

Using the classical sine transform \cite{deconinck2014},
\begin{align}
q(x,t)=\frac{2}{\pi} \int_0^{\infty} e^{-\omega(k)t} \sin(kx) \left[\sin(ky)q(y,0)dy-k\tilde{g}_0(\omega(k),t)\right]dk.
\label{solheat1}
\end{align}
The equivalence of the expressions is shown by deforming the contour of (\ref{solheat2}) back to the real line. The reason we do not work with (\ref{solheat1}) is twofold:
 \begin{enumerate}
\item Deforming the contour back to the real axis is possible only when classical transforms exist. Generally speaking, classical transforms do not exist for dispersive equations.
\item It is more straightforward to apply the method of steepest descent numerically to (\ref{solheat2}) than it is to (\ref{solheat1}).
\end{enumerate}
\subsection{Deformations of contours based on the method of steepest descent}
\label{sec_heatdeform}
We write the solution (\ref{solheat2}) as
\begin{align}
q(x,t)&=I_1+I_2+B_0,
\label{heatsolintg}
\end{align}
where
\begin{align*}
I_1&=\frac{1}{2\pi}\int_{-\infty}^{\infty} e^{ikx-\omega(k)t}\hat{q}_0(k)dk,\\
I_2&=-\frac{1}{2\pi}\int_{\partial D^+}e^{ikx-\omega(k)t} \hat{q}_0(-k)dk,\\
B_0&=-\frac{1}{2\pi}\int_{\partial D^+}e^{ikx-\omega(k)t} 2ik\tilde{g}_0(\omega(k),t)dk.
\end{align*}
The associated deformed contours for $I_1,I_2$ and $B_0$ will be defined by $\mathcal{C}_1^I,\mathcal{C}_2^I$ and $\mathcal{C}_0^B$ respectively in the following sections. In \cite{flyer2008}, for the FCM, the deformed contour $\tilde{\mathcal{C}}_0^B$ is independent of $(x,t)$, and is proposed for all three integrals $I_1,I_2$ and $B_0$. It is a hyperbola parameterized by $s\in \mathbb{R}$, shown in Figure \ref{heatcpb},
\begin{align}
k(s)=i \sin ( \pi/8 - i s).
\label{ffcontour}
\end{align} 
This contour $\tilde{\mathcal{C}}_0^B$ is also used in \cite{barros2019,fokas2009,papa2009} for different types of advection-diffusion equations. There are two major drawbacks of using $\tilde{\mathcal{C}}_0^B$: (i) the integrands of $I_1,B_0$ are not  defined on all of $\tilde{\mathcal{C}}_0^B$, and (ii) the evaluation of the integral along $\tilde{\mathcal{C}}_0^B$ quickly loses accuracy when $t$ increases as it does not follow the direction of steepest descent and large oscillations and potential growth destroy accuracy. To fix these issues with  FCMs, we use different deformations of the contours for $I_0,I_1$ and $B_0$ and the contours are deformed to follow the direction of steepest descent as much as possible.

\subsubsection{$I_1$: The integral involving $\hat{q}_0(k)$}
\label{sec_heatj1}
The phase function in the integrand is
\begin{align}
\theta(k;x,t)=ikx-\omega(k)t=ikx-k^2 t.
\end{align}
There is one saddle point $k_0=ix/2t$ where $\theta'(k_0;x,t)=0$ on the imaginary axis. The phase function $\theta(k;x,t)$ can be rewritten as
\[
\theta(k;x,t)= ikx-k^2 t = -t(k-ix/2t)^2-x^2/4t.
\]
The direction of steepest descent, along which the magnitude of $e^{\theta(k;x,t)}$ decays exponentially, is horizontal.
If Im$(k_0)=x/2t > \delta$,  the contour cannot be deformed to pass through the saddle point $k_0$ because the transform of the initial data $q \in C_{\delta}^{\infty} $ is only guaranteed to be defined for Im$(k)\leq \delta$. However, there is exponential decay in the integrand when the path is along the horizontal line Im$(k)= \delta$ since $t>0, x>0$. Hence the deformed path that we choose is a horizontal line $\mathcal{C}_1^I=\{k\in \mathbb{C}:\im{k}=h\}$, with $h=\min(\delta,x/2t)$.  
\begin{align*}
I_1 =-\frac{1}{2\pi}\int_{-\infty}^{\infty} e^{ikx-\omega(k)t}\hat{q}_0(k)dk=-\frac{1}{2\pi}\int_{\mathcal{C}_1^I} e^{ikx-k^2 t}\hat{q}_0(k)dk.
\end{align*}
The uniform convergence of Clenshaw-Curtis quadrature applied to $I_1$ for $x,t\geq c$ is established in Theorem~\ref{thmi1} (Appendix), after proper truncation and rescaling.

\begin{figure}
  \makebox[\textwidth][c]{
  \includegraphics[width=0.5\textwidth]{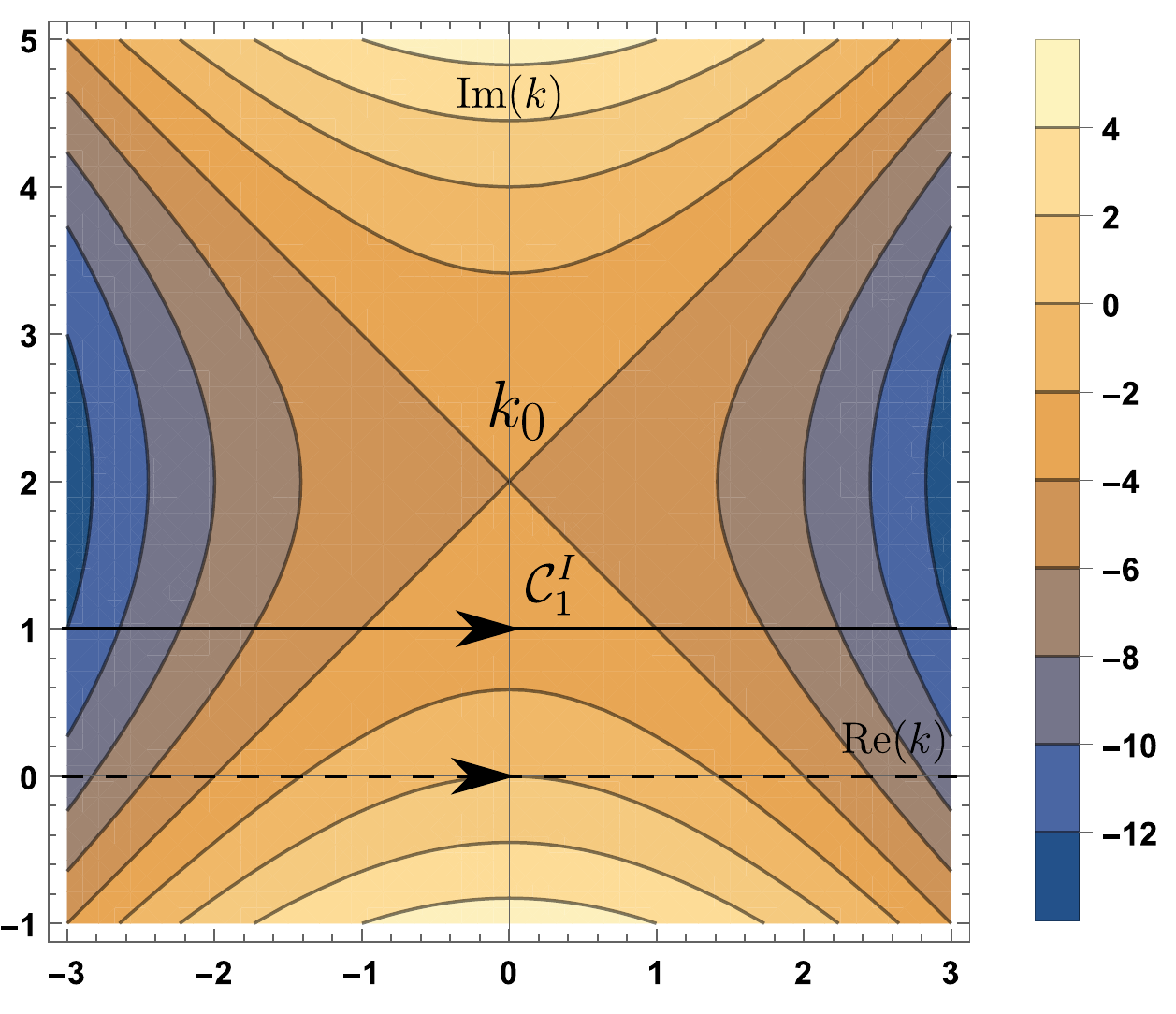}
  }
  \caption{The deformed horizontal contour $\mathcal{C}_1^I$ (solid) passing through $i$ with $\delta=1$, $x=4,t=1$, $k_0=2i$. The undeformed contour (dashed). The background contour plot shows the level sets of $\mbox{Re}(\theta(k,x,t))$. The integrand of $I_1$ is analytic for $\im{k}<1$ when $q_0\in C_1^{\infty}$.}
\end{figure}

\subsubsection{$I_2$: The integral involving $\hat{q}_0(-k)$}
\label{sec_heatj2}
Similar analysis can be applied to $I_2$ in (\ref{solheat2}). Here
\[
I_2=-\frac{1}{2\pi}\int_{\partial D^+}e^{ikx-\omega(k)t}\hat{q}_0(-k)dk.
\]
Because $\hat{q}_0(-k)$ is analytic and bounded for Im$(k)>-\delta$, we can deform the contour $\partial D^+$ to the horizontal line passing through $k_0=ix/2t$ defined by $\mathcal{C}_2^I=\{k\in\mathbb{C}: \im{k}=x/2t\}$, 
\[
I_2=-\frac{1}{2\pi}\int_{\mathcal{C}_2^I}e^{ikx-\omega(k)t}\hat{q}_0(-k)dk.
\]
The uniform convergence of Clenshaw-Curtis quadrature applied to $I_2$ for $x,t\geq c$ is established in Theorem~\ref{thmi1} (Appendix), after proper truncation and rescaling.
 
\begin{figure}
  \makebox[\textwidth][c]{
  \includegraphics[width=0.5\textwidth]{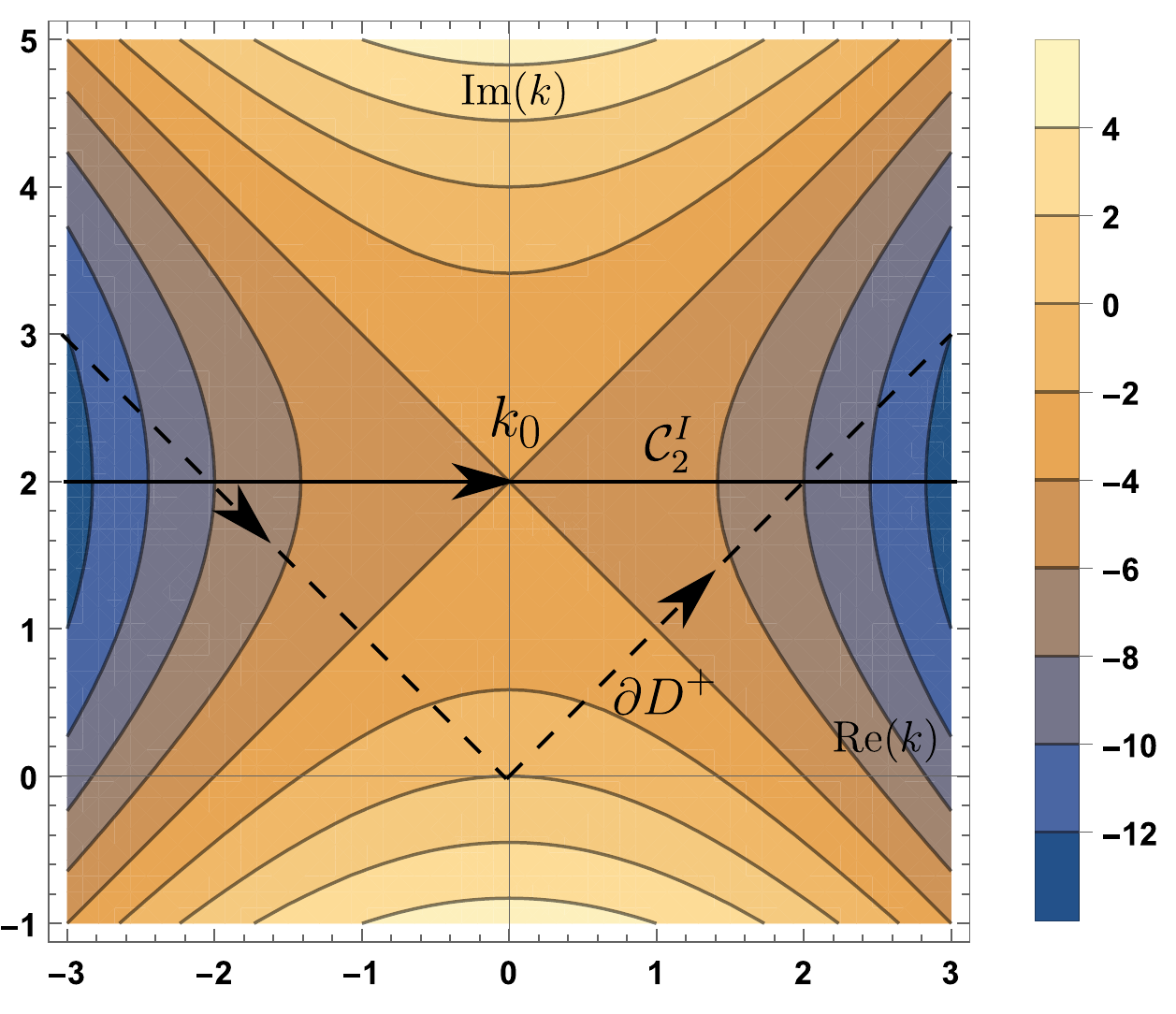}
  }
  \caption{The deformed horizontal contour for $I_2$ (solid) through $k_0=2i$ with $\delta=1$, $x=4,t=1$. The undeformed contour (dashed). The background contour  plot shows the level sets of $\mbox{Re}(\theta(k,x,t))$. The integrand of $I_2$ is analytic for $\im{k}>-1$. The dashed line is the undeformed contour $\partial D^+$.}
\end{figure}

\subsubsection{$B_0$: The integral of the transform of boundary data $\tilde{g}_0(\omega(k),t)$}
\label{sec_heati3}
The integral  $B_0$ in (\ref{solheat2}) containing the boundary data is more complicated compared to the integrals $I_1,I_2$.  There are two important factors that require special treatment:
\begin{enumerate}
\item The parameter $t$ appears both in the exponential and in the transformed boundary data $\tilde{g}(\omega(k),t)$ and therefore the phase $\theta(k,x,t)$ alone does not describe the decay of the integrand in $B_0$.
\item The evaluation of $e^{-k^2 t}\tilde{g}(\omega(k),t)$ is ill conditioned due to the oscillations and growth in $\tilde{g}(\omega(k),t)$ canceling those from the exponential.
\end{enumerate}

\paragraph{Example.}
To get a more concrete understanding, consider $g_0(t)=e^{-t}$. The transform is
\[
\tilde{g}_0(\omega(k),t)=\int_0^{t}e^{k^2 s}e^{-s}ds=\frac{1}{k^2-1}\left(e^{(k^2-1)t}-1\right).
\]
Since $g_0(s)=e^{-s}$ is bounded on the finite interval $0\leq s \leq t$, the transformed data $\tilde{g}_0(\omega(k),t)$ is an entire function of $k$ with removable poles at $k=\pm 1$. 
The integrand of $B_0$ contains two terms,
\begin{align}
e^{ikx-\omega(k)t} 2ik\tilde{g}_0(\omega(k),t)=\frac{ike^{ikx-t}}{\pi(k^2-1)}-\frac{ike^{ikx-k^2t}}{\pi(k^2-1)}.
\label{twoterm1}
\end{align}
The second term follows the horizontal direction of steepest descent but the first term is not exponentially localized on  horizontal lines in the complex $k$-plane. Although the integral of the first term on the $\partial D^+$ is zero, it is not possible to separate the two terms, in general, for all $k$. 
\paragraph{General case.}
We write the transform $\tilde{g}_0(k,t)$ as
\[
\tilde{g}_0(\omega(k),t)=\int_0^{t}e^{k^2 s}g_0(s)ds=-\int_{t}^{\infty}e^{k^2 s}g_0(s)ds+\int_0^{\infty}e^{k^2 s}g_0(s)ds,
\]
for $k \in D^+$.
Therefore the integrand in $B_0$ is
\begin{align}
e^{ikx -k^2 t}\tilde{g}_0(\omega(k),t)=-e^{ikx}\int_{0}^{\infty}e^{k^2 s}g_0(s+t)ds + e^{ikx -k^2 t}\tilde{g}_0(\omega(k),\infty).
\label{twoterm2}
\end{align}
The two terms on the right-hand side of (\ref{twoterm2}) behave the same as the two terms in (\ref{twoterm1}). Because $\tilde{g}_0(\omega(k),\infty)$ is in general not defined outside $D^+$, a separation only exists inside $D^+$. Without splitting the two terms, to get exponential decay for both terms, the contour $\partial D^+$ is deformed to $\mathcal{C}_0^I$ passing through the saddle point horizontally and turns up when the second term in the integrand is negligible, see Figure \ref{heatcontour3}. 
The corner point $k_1=\pm L+ix/2t$ is determined by $L=\max(L_1,\sqrt{\gamma})$ with specified tolerance $\epsilon$ where $\abs{e^{-L_1^2 t}}=\epsilon$ characterizes the exponential decay and $\sqrt{\gamma}$ allows the oblique segment to be away from $k=0$. With this choice of contour, the exponential part in the second term decays exponentially along the horizontal segment and keeps the same magnitude along the oblique segment while the exponential part in the first term keeps the same magnitude along the horizontal segment and decays exponentially along the oblique segment. Uniform accuracy is shown in Theorem \ref{thmi3} (Appendix) after proper truncation and rescaling.

\begin{figure}
  \makebox[\textwidth][c]{
  \includegraphics[width=0.5\textwidth]{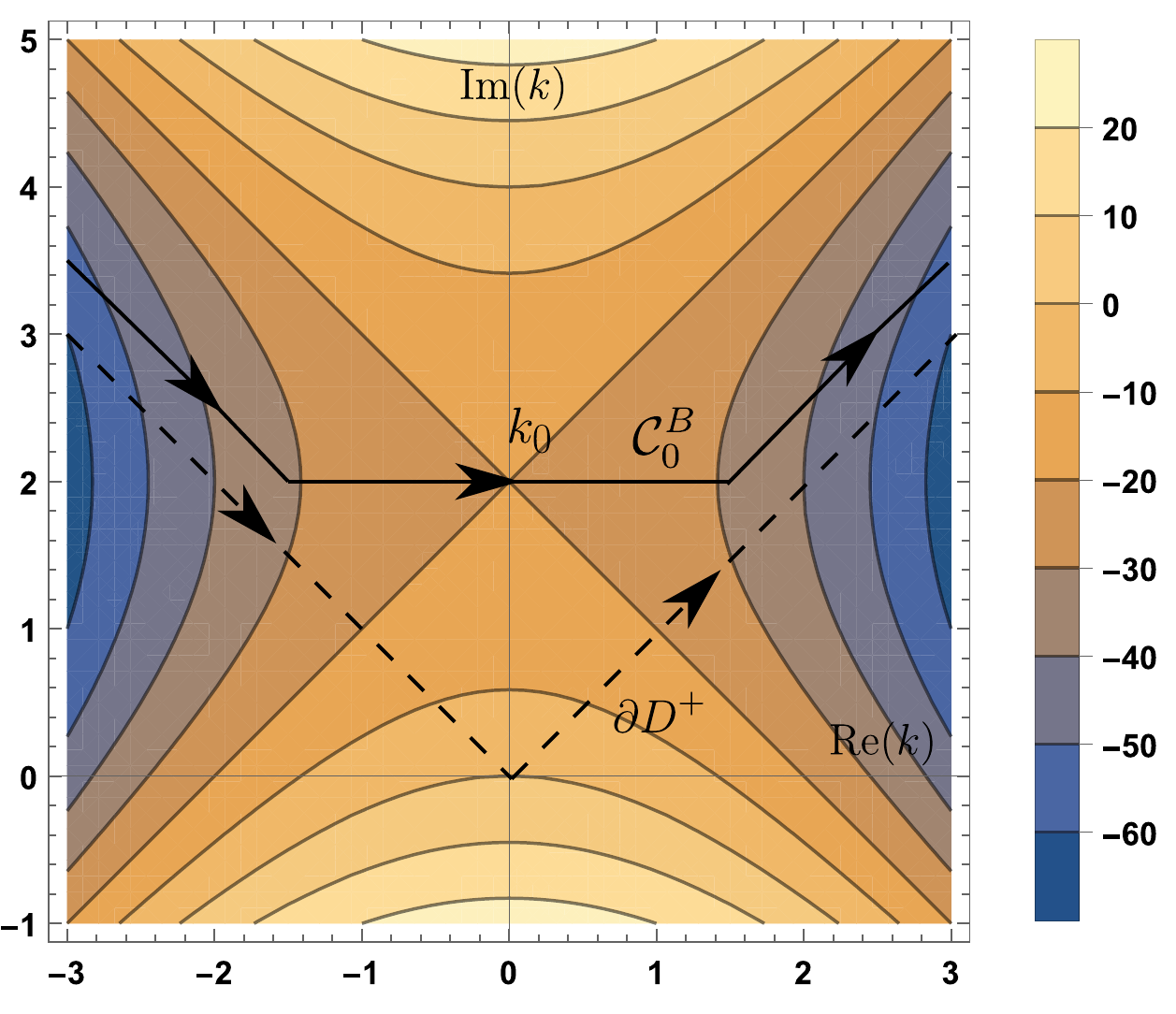}
  }
  \caption{The undeformed contour $\partial D^+$. The solid line gives the deformed contour  $\mathcal{C}_0^I$ of $I_0$ that goes through $k_0=2i$ and turns to rays parallel to $\partial D^+$ with $x=20,t=5$.  The integrand of $I_0$ is entire. The background contour  plot shows the level sets of $\mbox{Re}(\theta(k,x,t))$.}
  \label{heatcontour3}
\end{figure}

\subsection{A numerical example}
\label{exheat}
With these deformed contours, we can numerically evaluate the integrals efficiently for arbitrarily large  values of the parameters $x,t$. Figure \ref{heatsol} shows the solution to the heat equation with initial and boundary data $q_0(x)=e^{-x}$, $g_0(t)=e^{-t}$. Although exact transforms can be obtained for this choice of data, they are only used for computing the errors and our NUTM does not make use of the formulas.

To demonstrate the uniform accuracy for large $x,t$, we plot the absolute errors $E_{\text{NUTM}}$ and $E_{\text{FCM}}$ along 3 different curves (a) $t=0.1$, (b) $x=0.1$, and (c) $t=x^2$ in Figure \ref{heaterr}. The error $E_{\text{NUTM}}$ is obtained using the contours $\mathcal{C}_1^I,\mathcal{C}_2^I$ and $\mathcal{C}_0^B$. The error $E_{\text{FCM}}$ is obtained using the contour  $\tilde{\mathcal{C}}_0^B$ in (\ref{ffcontour}) \cite{flyer2008}. The initial and boundary conditions are $q_0(x)=e^{-10x}$, $g_0(t)=e^{-10t}$ to allow deformation in a larger region.  The number of collocation points $N=120$ is the same for both methods. This is a coarse grid for the integrals with the errors approximately $10^{-3}$ when $s=0.1$ is small but it shows the efficiency of the NUTM as $s$ grows. The true solution is computed using Mathematica's built-in numerical integration routine \verb|NIntegrate| along the undeformed contour $\partial D^+$ with sufficient recursions and precision. This is time consuming if the transforms of the initial and boundary data need to be computed. The truncation tolerance is set to $10^{-13}$ for determining the truncation of the deformed path. This value of the truncation tolerance is chosen so that it is small enough to show the trend of the errors when $x,t$ vary and the truncation is not affected by the rounding errors. These settings are the same for other examples in the paper unless stated otherwise.

The absolute error $E_{\text{NUTM}}$ decreases in all cases as $x,t$ grow while $E_{\text{FCM}}$ grows when $t$ increases. This can be explained simply by the fact that the contour used in the FCM does not follow the steepest descent path. Furthermore, even when $t$ is fixed in Figure~\ref{heaterr}(a), $E_{\text{FCM}}$ decreases slower than $E_{\text{NUTM}}$. On the other hand, $E_{\text{NUTM}}$ increases relative to the true solution. This is mainly due to the fact that the magnitude of the solution is smaller than the truncation tolerance for $x>5$ at which point the numerical solution has almost all contours truncated. In Figure~\ref{heaterr}(b-c), $E_{\text{NUTM}}$ maintains good relative accuracy. In Figure \ref{heaterr}(c), $E_{\text{NUTM}}$ starts with a larger error because $t=s^2=0.01$ is very small and close to the initial condition which requires more nodes to produce the same order of errors compared with the other two starting from $t=0.1$.      
\begin{figure}
  \makebox[\textwidth][c]{
  \includegraphics[width=0.8\textwidth]{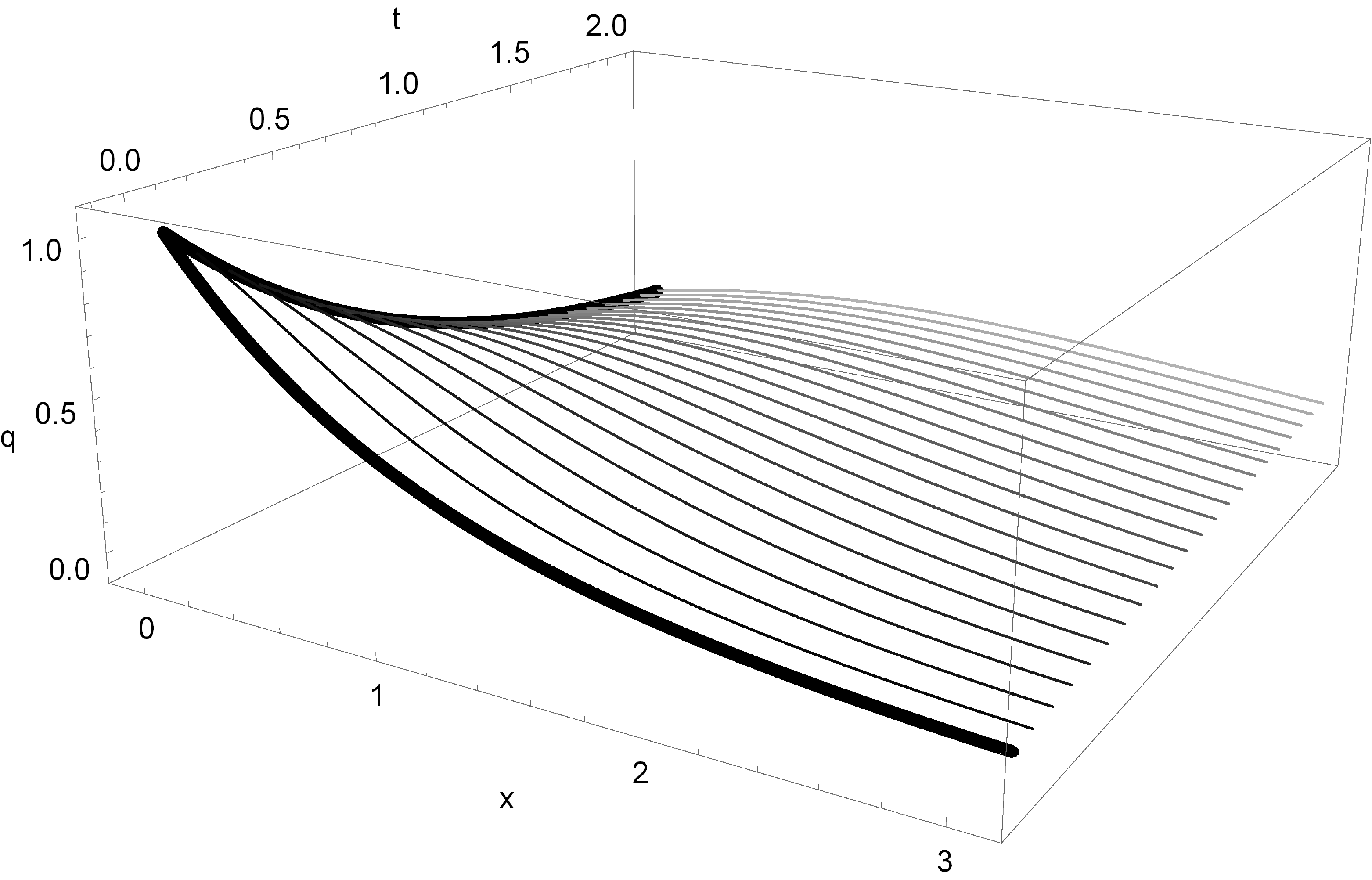}
  }
  \caption{The solution to the heat equation (\ref{heateq}) with exponential decay initial and boundary data $q(x,0)=e^{-x}$,  $q(0,t)=e^{-t}$. The bold curves are the initial and boundary conditions.}
  \label{heatsol}
\end{figure}

\begin{figure}
  \makebox[\textwidth][c]{
  \begin{subfigure}{0.3\textwidth}
  \includegraphics[width=\textwidth]{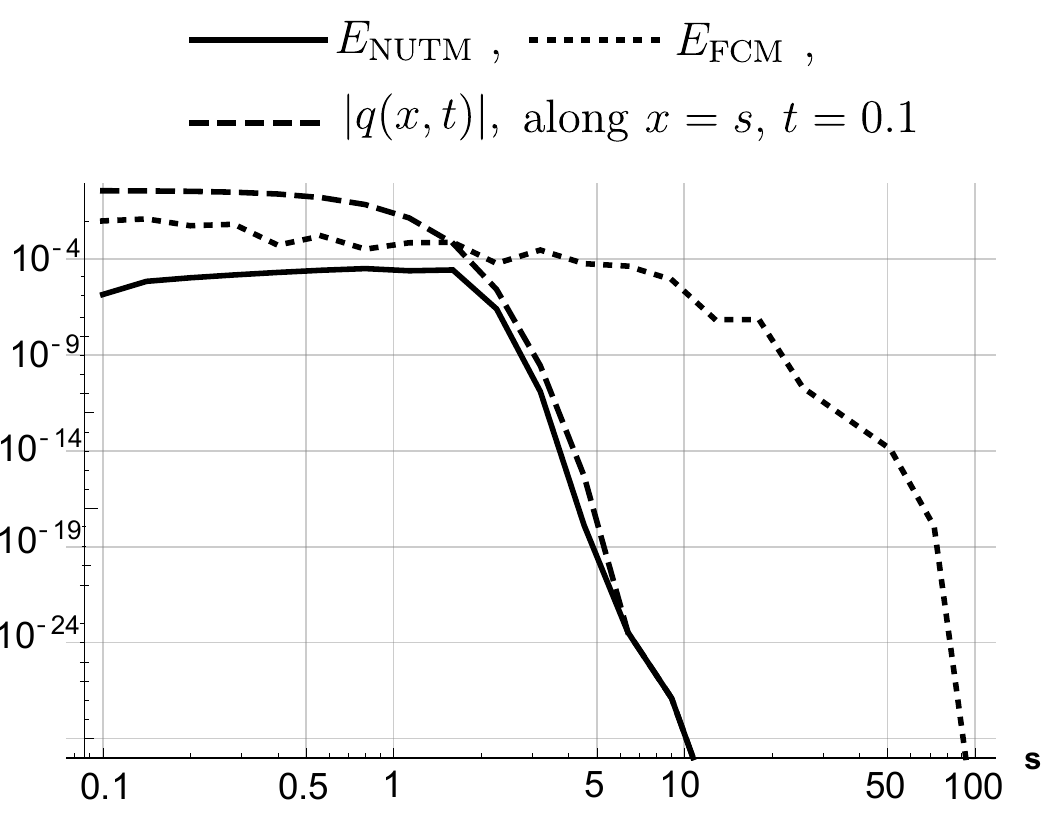}
  \caption{}
  \end{subfigure}
  \begin{subfigure}{0.3\textwidth}
  \includegraphics[width=\textwidth]{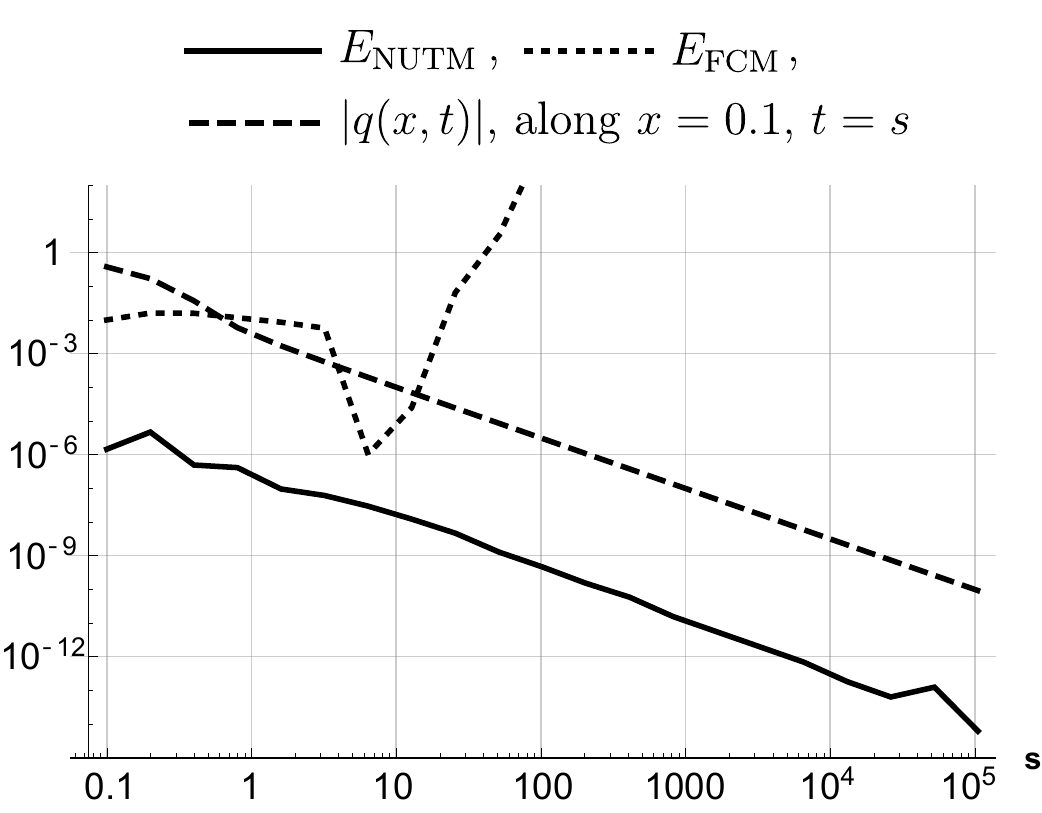}
  \caption{}
  \end{subfigure}
 \begin{subfigure}{0.3\textwidth}
  \includegraphics[width=\textwidth]{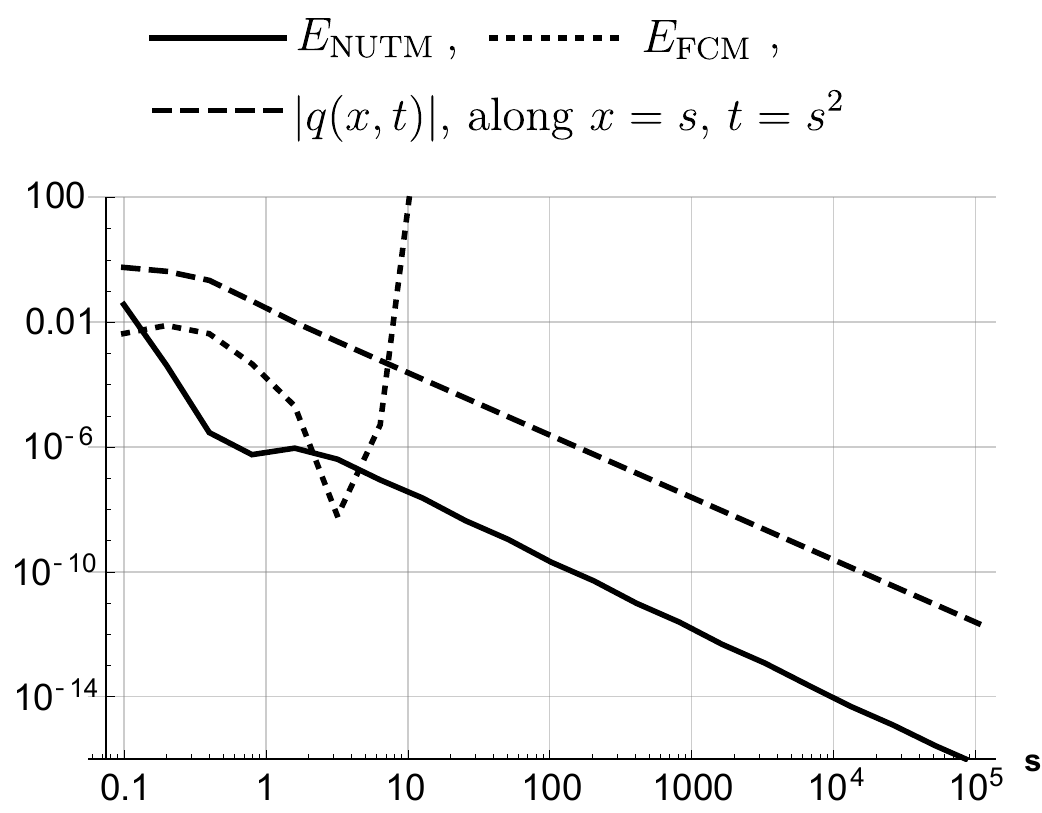}
  \caption{}
  \end{subfigure}

  }
  \caption{The absolute errors of the numerical solution to the heat equation with initial condition $q_0(x)=e^{-10 x}$ and boundary condition $g_0(t)=e^{-10 t}$ along (a) $x=s,t=0.1$, (b) $x=0.1, t=s$, (c) $x=s,t=s^2$ for $s\in [0.1,10^5]$. The error $E_{\text{NUTM}}$ is obtained using the contours $\mathcal{C}_1^I,\mathcal{C}_2^I$ and $\mathcal{C}_0^B$. The error $E_{\text{FCM}}$ is obtained using the contour $\tilde{\mathcal{C}}_0^B$ in Figure \ref{heatcontour3}. The absolute value of the solution $\abs{q(x,t)}$ is also plotted with dashed lines for reference. The FCM loses accuracy as $t$ grows while $E_{\text{NUTM}}$ decreases in all cases as parameters increase. }
  \label{heaterr}
\end{figure}

\paragraph{Remark.}
As can be seen in (\ref{twoterm1}) and (\ref{twoterm2}), there is large cancellation in the exponentials. To avoid  potential overflow/underflow problems, we use $\hat{g}_0(\omega(k),T)$ defined by   
\begin{align}
\hat{g}_0(\omega(k),T)=e^{-\omega(k)T}\tilde{g}_0(\omega(k),T)=\int_0^T e^{\omega(k) (s-T)}g_0(s)ds.
\end{align}

\section{The linear Schr\"odinger equation on the half-line}

\label{ch_lnls}
Next, we consider a dispersive example, the linear Schr\"odinger (LS) equation:
\begin{align}
iq_t=-q_{xx},\,\,\,\,  x>0,\,\,t>0,
\label{lseq}
\end{align}
with Dirichlet boundary data $g_0 \in C^{\infty}_\gamma$ and initial data $q_0 \in C^{\infty}_\delta$. 
\subsection{The solution formula from the unified transform method}
The dispersion relation for (\ref{lseq}) is $\omega(k)=ik^2$. Define the transform of the initial data $\hat{q}_0(x)$ and the transform of the Dirichlet boundary data $\tilde{g}_0(t)$ by (\ref{qhat}) and (\ref{gtilde}).
The UTM provides the solution in terms of the following contour integrals \cite{deconinck2014},
\begin{align}
q(x,t)=\frac{1}{2\pi}\int_{-\infty}^{\infty} e^{ikx-\omega(k)t}\hat{q}_0(k)dk-\frac{1}{2\pi}\int_{\partial D^+}e^{ikx-\omega(k)t}\left[\hat{q}_0(-k)-2k\tilde{g}_0(\omega(k),t)\right]dk,
\label{solls}
\end{align}
where the contour $\partial D^+$ is the positively oriented boundary of the first quadrant $D^+=\{k\in \mathbb{C}: \re{k}\geq 0, \im{k}\geq 0  \}$.
With the assumption of the decay of $g_0(t)$, the contour can be deformed to the lower-half plane inside $\tilde{D}=\{k\in \mathbb{C}:\re{k^2}<\gamma \}$ as in Figure \ref{lscontour}. But this is not enough to completely eliminate the oscillations. In general, other methods for oscillatory integrals are required when $t$ is not sufficiently large or the saddle point $k_0$ has large modulus.  
\begin{figure}
  \makebox[\textwidth][c]{
  \begin{subfigure}{0.45\textwidth}
    \includegraphics[width=\textwidth]{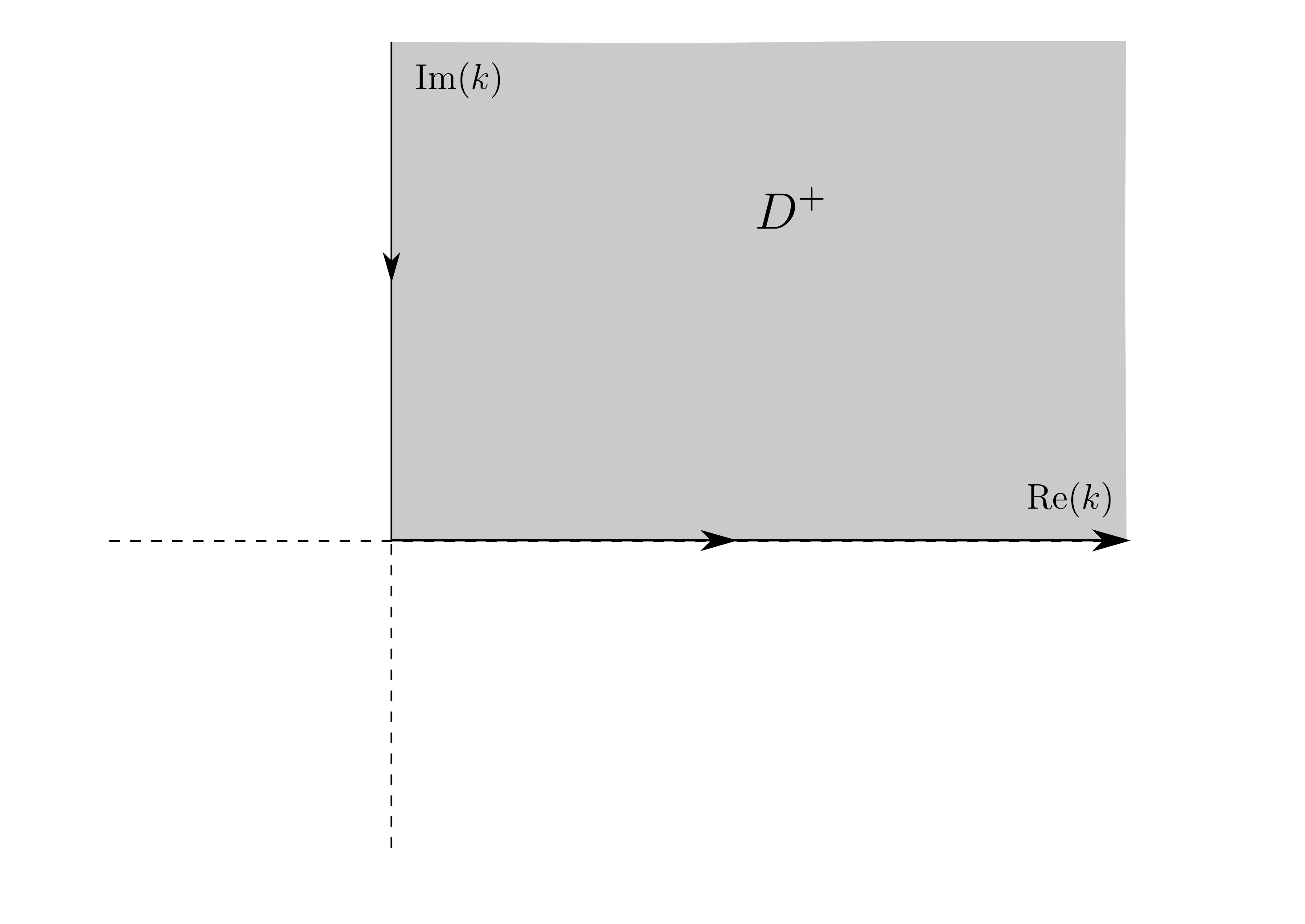}
    \caption{}
    \end{subfigure}
  \begin{subfigure}{0.45\textwidth}
        \includegraphics[width= \textwidth]{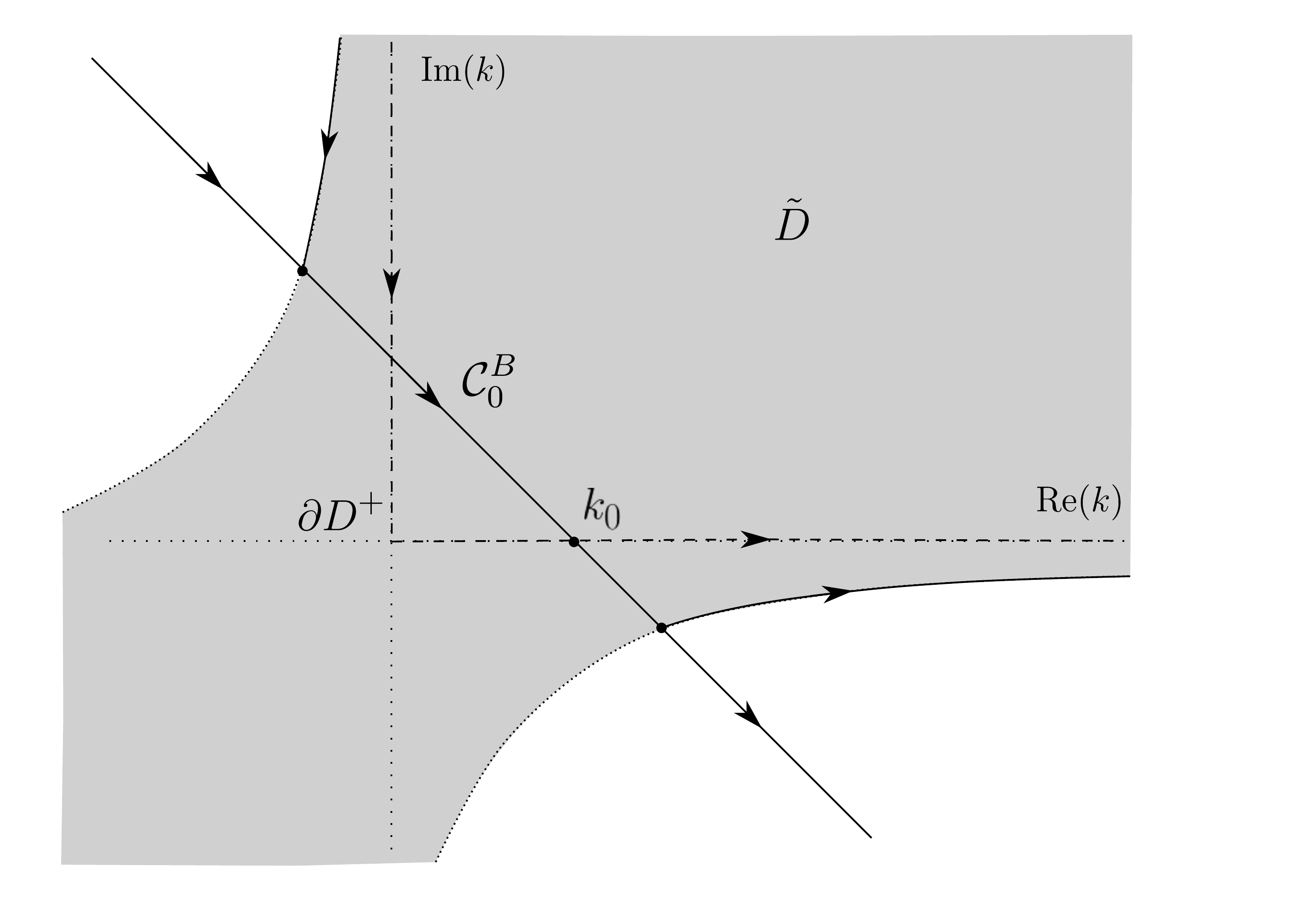}
    \caption{}
    \end{subfigure}

  }
  \caption{Regions for the LS equation. Panel (a) shows the region $D^+=\{k\in \mathbb{C}^+:\re{k^2}<0 \}$. Panel (b) shows  $\tilde{D}=\{k\in \mathbb{C}:\re{k^2}<\gamma \} $ and a schematic of the deformed contour from $\partial D^+$ to $\mathcal{C}_0^{B}$ for $B_0$ in (\ref{lssoleq}), see Section \ref{sec_lsi0} for details of the deformation. }
  \label{lscontour}
\end{figure}
\subsection{Deformations of the contours based on the method of steepest descent}
We separate the different integrals in the solution formula (\ref{solls}),
\begin{align}
q(x,t)=I_0+I_1+B_0,
\label{lssoleq}
\end{align}
where
\begin{align*}
I_0&=\frac{1}{2\pi}\int_{-\infty}^{\infty} e^{ikx-\omega(k)t}\hat{q}_0(k)dk,\\
I_1&=-\frac{1}{2\pi}\int_{\partial D^+}e^{ikx-\omega(k)t}\hat{q}_0(-k)dk,\\
B_0&=\frac{1}{2\pi}\int_{\partial D^+}e^{ikx-\omega(k)t}2k\tilde{g}_0(\omega(k),t)dk.
\end{align*}

\subsubsection{$I_1$: integral with the transform of the initial data}
\label{sec_lsj1}
The phase function in $I_1$ is
\begin{align}
\theta(k;x,t)=ikx-\omega(k)t=ikx-ik^2 t.
\end{align}
There is one saddle point $k_0=x/2t$ on the positive real axis satisfying $\theta'(k_0;x,t)=0$. Near the saddle point~$k_0$,
\[
\theta(k;x,t)= ikx-ik^2 t = -it(k-x/2t)^2+ix^2/4t.
\]
The directions of steepest descent are $-\pi/4$ and $3\pi/4$. Similar to the case of the heat equation, the transformed initial data $\hat{q}_0(k)$ is bounded and analytic in $\im{k} < \delta$ when $q_0 \in C^{\infty}_{\delta}$. Hence we choose the deformed contour $\mathcal{C}_1^I=\{a+k_0+i b : a \in (-\infty, - \delta), b=\delta\} \cup \{ a+k_0-i a  : a \in [-\delta, \infty)\}$ to be a horizontal ray with height Im$(k)=\delta$ and a straight-line segment with slope $-1$ passing through the saddle point as shown in Figure \ref{lsconp1}. The integral $I_1$ becomes
\[
I_1=\frac{1}{2\pi}\int_{\mathcal{C}_1^I} e^{ikx-ik^2 t}\hat{q}_0(k)dk.
\]
\begin{figure}
  \makebox[\textwidth][c]{
  \includegraphics[width=0.5\textwidth]{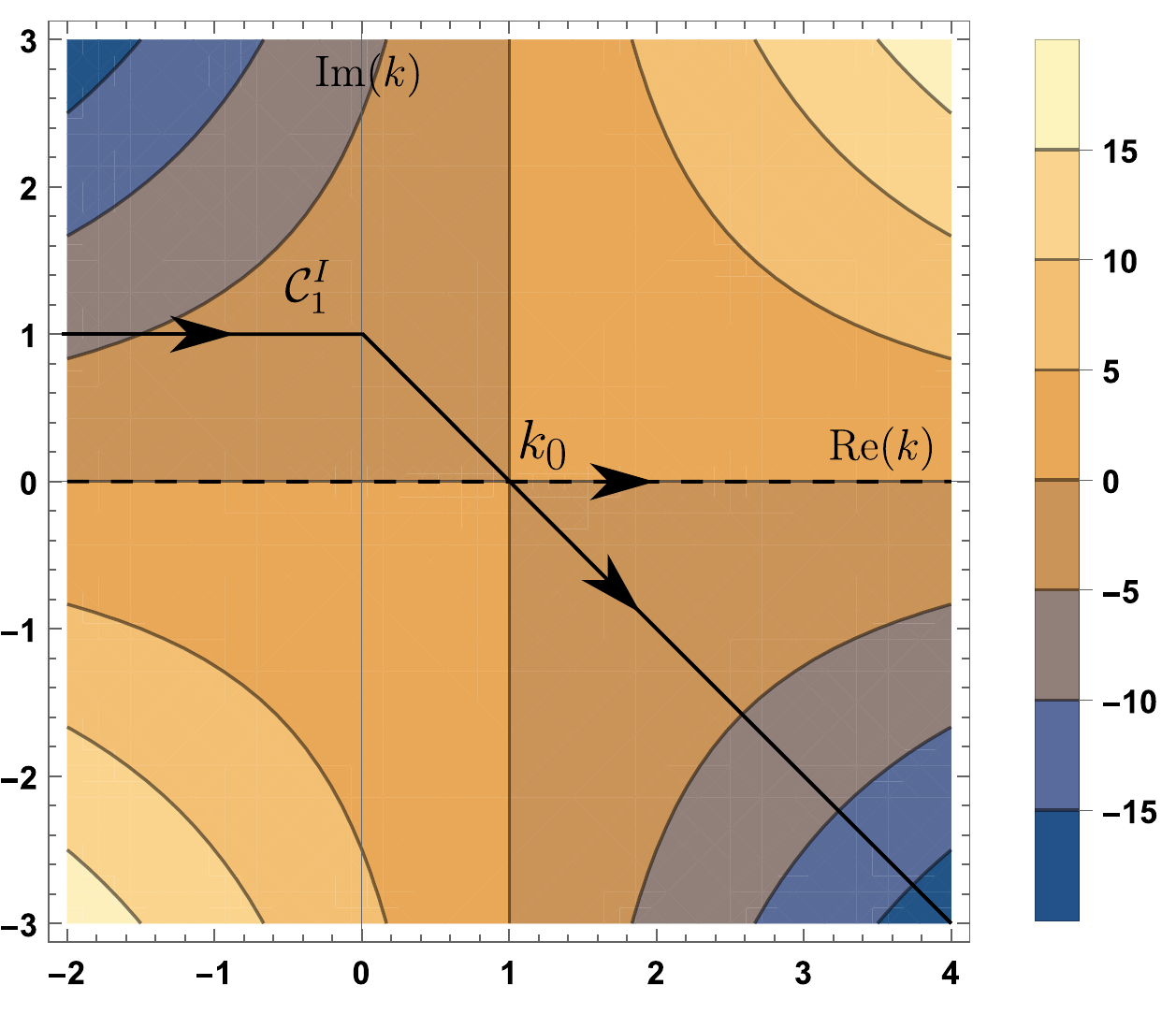}
  }
  \caption{The undeformed contour (dashed). The deformed horizontal contour $C_1^I$ (solid) going through $k_0=1$, $\delta=1$, $x=2,t=1$. The background contour  plot shows the level sets of $\mbox{Re}(\theta(k,x,t))$. The integrand of $I_1$ is analytic for $\im{k}<1$.}
  \label{lsconp1}
\end{figure}

\subsubsection{$I_2$: integral with the transform of the initial data $\hat{q}_0(-k)$}
\label{sec_lsj2}
Similar analysis can be applied to $I_2$ with $\hat{q}_0(-k)$ in (\ref{solls}). Since the transform $\hat{q}_0(-k)$ is analytic and bounded for $\im{k}>-\delta$, we can deform the contour $\partial D^+$ to  
\[
\mathcal{C}_2^I=\{a+k_0-i a : a \in (-\infty,  \delta)\} \cup \{ a+k_0-i b  : a \in [\delta, \infty), b=\delta\},\]
see Figure \ref{lsconp2}. Therefore, $I_2$ becomes 
\[
I_2=-\frac{1}{2\pi}\int_{\mathcal{C}_2^I } e^{ikx-ik^2 t}\hat{q}_0(-k)dk.
\]
\begin{figure}
  \makebox[\textwidth][c]{
  \includegraphics[width=0.5\textwidth]{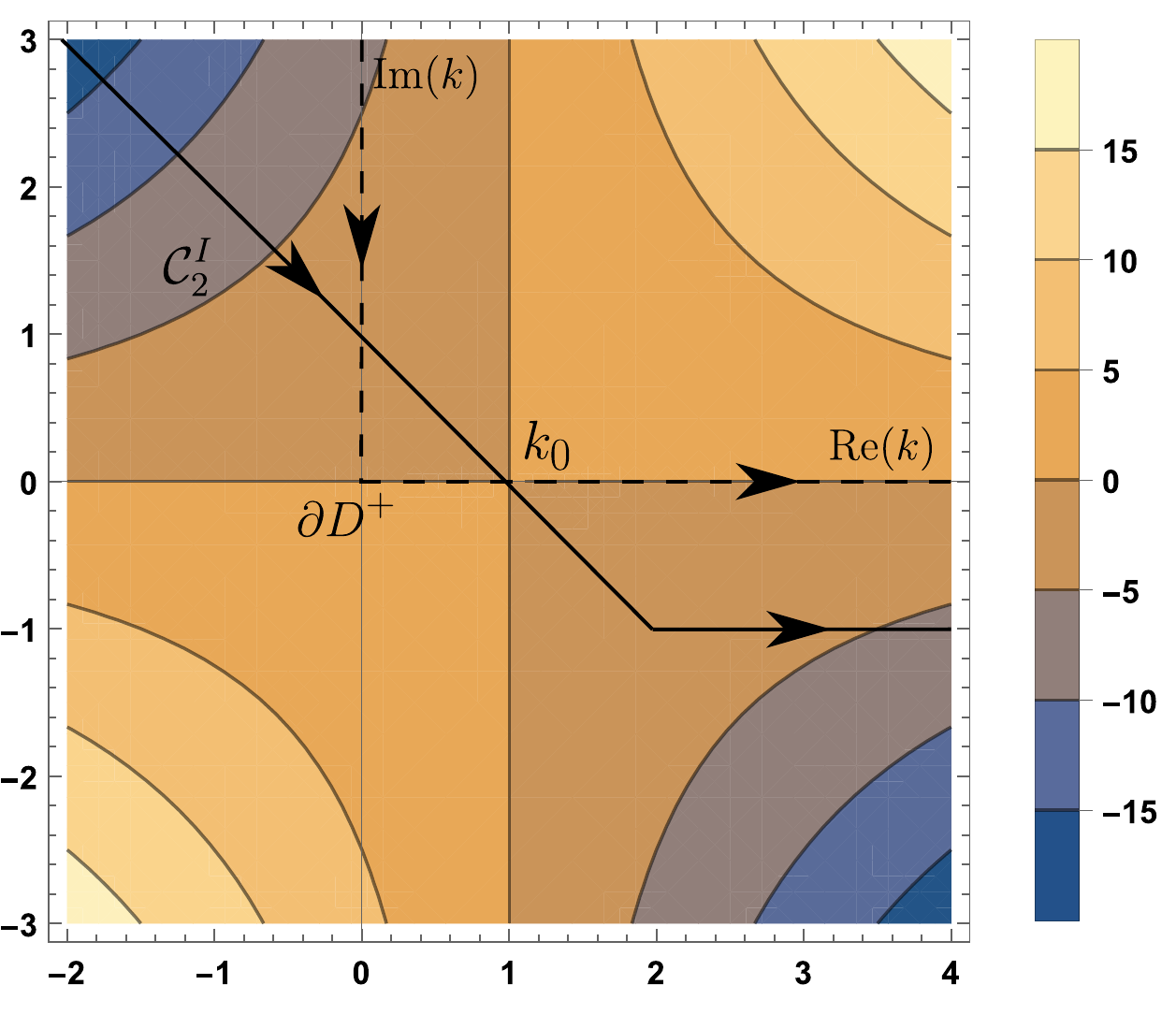}
  }
  \caption{The undeformed contour (dashed). The deformed contour for $I_2$ (solid) going through $k_0=1$, $\delta=1$, $x=2,t=1$. The background contour plot shows the level sets of $\mbox{Re}(\theta(k,x,t))$. The integrand of $I_2$ is analytic for $\im{k}>-1$. }
  \label{lsconp2}
\end{figure}

\subsubsection{$B_0$: integral of the transform of boundary data $\tilde{g}_0(\omega(k),t)$}
\label{sec_lsi0}
The issues discussed in Section \ref{sec_heati3} also appear in the case of the LS equation. However, now the region where we can deform the contour is restricted. The same decomposition as in (\ref{twoterm2}) gives 
\begin{align}
e^{ikx -k^2 t}\tilde{g}_0(\omega(k),t)=-e^{ikx}\int_{0}^{\infty}e^{k^2 s}g_0(s+t)ds + e^{ikx -k^2 t}\tilde{g}_0(\omega(k),\infty).
\label{twotermls}
\end{align}
For generic $g_0(t)$,  if the contour of $B_0$ is along the $-\pi/4$ direction at the saddle point $k_0=x/2t$, the first term in (\ref{twotermls}) grows exponentially as $x\rightarrow \infty$ since $\re{ikx}>0$. On the other hand, $\tilde{g}_0(\omega(k),\infty)$ may not be extendable outside $D^+$. With the assumption that $g_0\in C^{\infty}_{\gamma}$, it becomes possible to deform the path to the lower-half plane to obtain some exponential decay. The steps of the deformation are: 
\begin{enumerate}
\item The region $D^+$ is extended to $\tilde{D}$. The transformed data $\tilde{g}_0(\omega(k), \infty)$ is analytic in $\tilde{D}$, and continuous up to  $\partial \tilde{D}$.
\item The contour $\partial D^+$ is deformed to $\mathcal{C}_{0,a}^{B} \cup \mathcal{C}_{0,b}^{B} \cup \mathcal{C}_{0,c}^{B}$ as shown in Figure \ref{lsconp3} where $ \mathcal{C}_{0,b}^{B}$ is the straight-line segment passing through the saddle point along the steepest-descent direction up to $\partial \tilde{D}$ and $\mathcal{C}_{0,a}^{B}, \mathcal{C}_{0,c}^{B}$ are the unbounded curved segments along $\partial \tilde{D}$.
\item Using that $e^{ikx-\omega(k)t} \int_t^{\infty}e^{\omega(k) s}g_0(s)ds$ is bounded and analytic in $\tilde{D}$, we can replace $\tilde{g}_0(\omega(k),t)$ with $\tilde{g}_0(\omega(k),\infty)$ using Jordan's lemma,
\begin{align}
B_0=\frac{1}{2\pi}\int_{\partial D^+}e^{ikx-\omega(k)t} 2k\tilde{g}_0(\omega(k),t) dk=\frac{1}{2\pi}\int_{\mathcal{C}_{0,a}^{B}\cup \mathcal{C}_{0,b}^{B}\cup \mathcal{C}_{0,c}^{B}}e^{ikx-\omega(k)t}2k\tilde{g}_0(\omega(k),\infty)dk. 
 \label{i0ginf}
 \end{align} 
\item The integral along $\mathcal{C}_{0,a}^{B}$ is decomposed into two parts to maximize decay along the steepest descent direction:
 \begin{align*}
\int_{\mathcal{C}_{0,a}^{B}}e^{ikx-\omega(k)t}2k\tilde{g}_0(\omega(k),\infty)dk =\int_{\mathcal{C}_{0,d}^{B}}e^{ikx-\omega(k)t}2k\tilde{g}_0(\omega(k),t)dk+\int_{\mathcal{C}_{0,a}^{B}}e^{ikx-\omega(k)t}2k\tilde{g}^c_0(\omega(k),t)dk,  
 \end{align*} 
 where 
\[
\tilde{g}^c_0(\omega(k),t)=\int_t^{\infty} e^{\omega{k}s} g_0(s)ds ,
\]
is the complementary transform of $g_0$.
\item The integral along $\mathcal{C}_{0,c}^{B}$ is decomposed into two parts: 
 \begin{align*}
\int_{\mathcal{C}_{0,c}^{B}}e^{ikx-\omega(k)t}2k\tilde{g}_0(\omega(k),\infty)dk=& 
\int_{\mathcal{C}_{0,c}^{B}}e^{ikx-\omega(k)t}2\left(k\tilde{g}_0(\omega(k),\infty)-k_0\tilde{g}_0(\omega(k_0),\infty)\right)dk\\
&+\int_{\mathcal{C}_{0,e}^{B}}e^{ikx-\omega(k)t}2k_0\tilde{g}_0(\omega(k_0),\infty)dk.
\end{align*} 
The second integral on the right-hand side is deformed to follow the direction of steepest descent. 
\item With the above steps, we obtain
\begin{align*}
 B_0=&\frac{1}{2\pi}\int_{\mathcal{C}_{0,a}^{B}}e^{ikx-\omega(k)t} 2k\tilde{g}^c_0(\omega(k),t) dk+\frac{1}{2\pi}\int_{\mathcal{C}_{0,d}^{B}}e^{ikx-\omega(k)t} 2k\tilde{g}_0(\omega(k),t dk\\
 &+\frac{1}{2\pi}\int_{\mathcal{C}_{0,b}^{B}}e^{ikx-\omega(k)t} 2k\tilde{g}_0(\omega(k),\infty) dk \\
 &+\frac{1}{2\pi}\int_{\mathcal{C}_{0,c}^{B}}e^{ikx-\omega(k)t}\left[2k\tilde{g}_0(\omega(k),\infty)-2k_0\tilde{g}_0(\omega(k_0),\infty)\right]dk\\
 &+\frac{1}{2\pi}\int_{\mathcal{C}_{0,e}^{B}}e^{ikx-\omega(k)t} 2k_0\tilde{g}_0(\omega(k_0),\infty) dk.
\end{align*} 
\end{enumerate}
Using the deformed contour, for large $x,t$ , the integral is exponentially localized near the saddle point on $\mathcal{C}_{0,b}^{B}$. When the integrand is not sufficiently small near the endpoints of $\mathcal{C}_{0,b}^{B}$, the oscillations in the integrand along  $\mathcal{C}_{0,c}^{B}$ and $\mathcal{C}_{0,d}^{B}$ become important. Most of the potential error comes from the integral along $\mathcal{C}_{0,c}^{B}$ as the integrand along $\mathcal{C}_{0,d}^{B}$ has exponential decay from the $e^{ikx}$ factor. The contour $\mathcal{C}_{0,c}^{B}$ asymptotically approaches the real axis. We use the Levin collocation method \cite{iserles2006} for the integrals along $\mathcal{C}_{0,c}^{B}$ and $\mathcal{C}_{0,d}^{B}$ to maintain accuracy for large $x,t$. The rest of the integrals in $B_0$, as well as those making up $I_1$ and $I_2$, are computed using Clenshaw-Curtis quadrature. 
\subsection{A numerical example}
Consider the initial condition $q_0(x)=0$, and the Dirichlet boundary condition $g_0(t)=te^{-t}$. The real part of the solution to (\ref{lseq}) with this choice of data is shown in Figure \ref{lssol}. Dispersive waves quickly emerge from the boundary, becoming more oscillatory for large $x$. The absolute error and the magnitude of the solution evaluated along (a) $t=0.1$, (b) $x=0.1$, (c) $t=x^2$ are shown in Figure \ref{lserr}. The errors shown in dotted curves are computed with $N=20$ collocation points for each part of the contour in $B_0$ while the errors shown in solid curves are computed with $N=40$ collocation points. The absolute errors decrease as $x,t$ increase.
In Figure \ref{lserr}(a), we see that although the initial condition is zero, at $t=0.1$ the solution $q(x,t)$ only decreases algebraically. This makes traditional time-stepping method inefficient even if we ignore issues related to the highly oscillatory nature of the solution.    
\begin{figure}
  \makebox[\textwidth][c]{
  \includegraphics[width=0.5\textwidth]{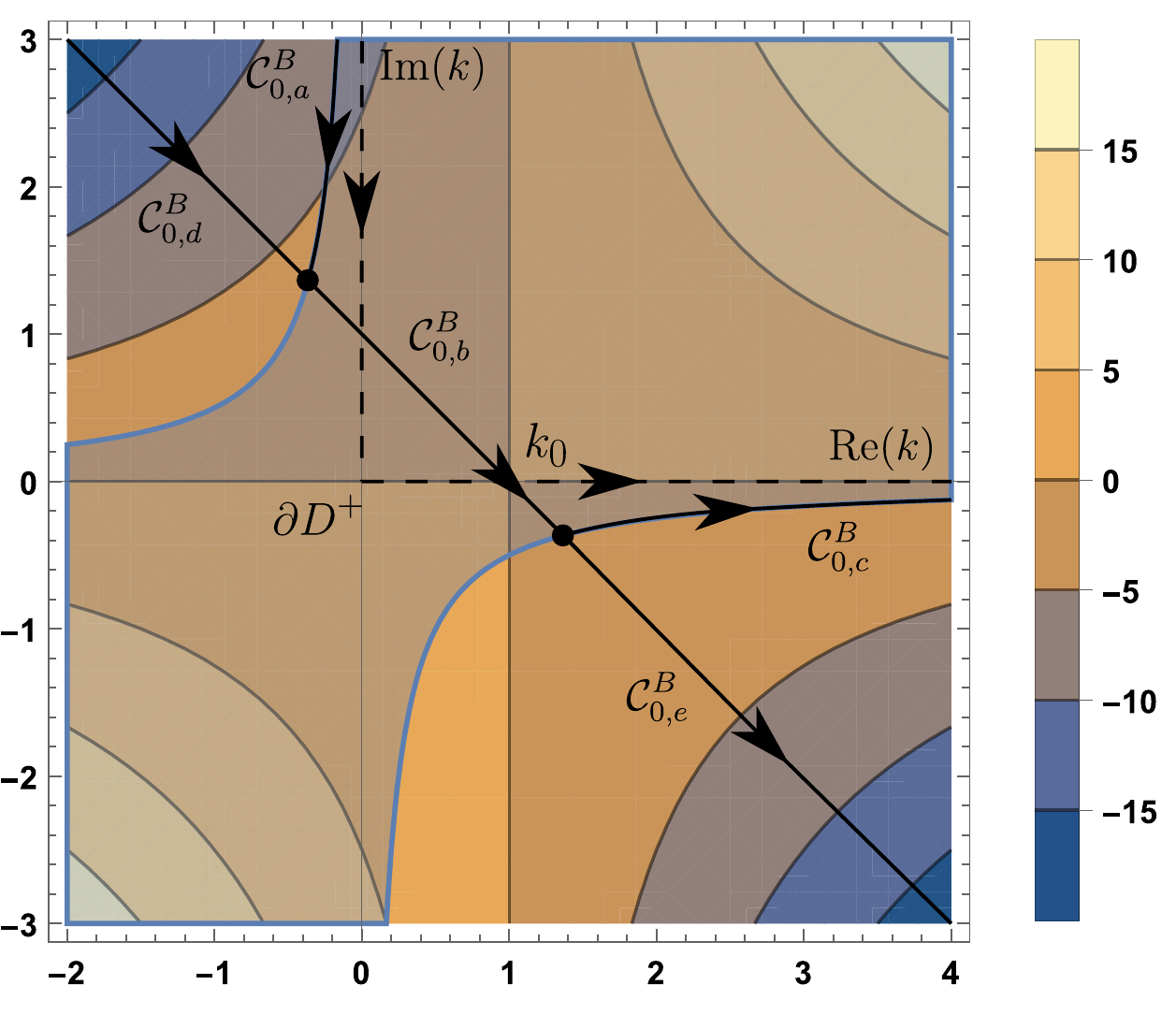}
  }
  \caption{The undeformed contour (dashed). The deformed contour for $B_0$ (solid) through $k_0=1$ with $x=2,t=1$,$\gamma=1$, see Section \ref{sec_lsi0} for details of the deformation. The background contour plot shows the level sets of $\mbox{Re}(\theta(k,x,t))$. }
  \label{lsconp3}
\end{figure}

%
%

\begin{figure}
  \makebox[\textwidth][c]{
  \includegraphics[width=0.8\textwidth]{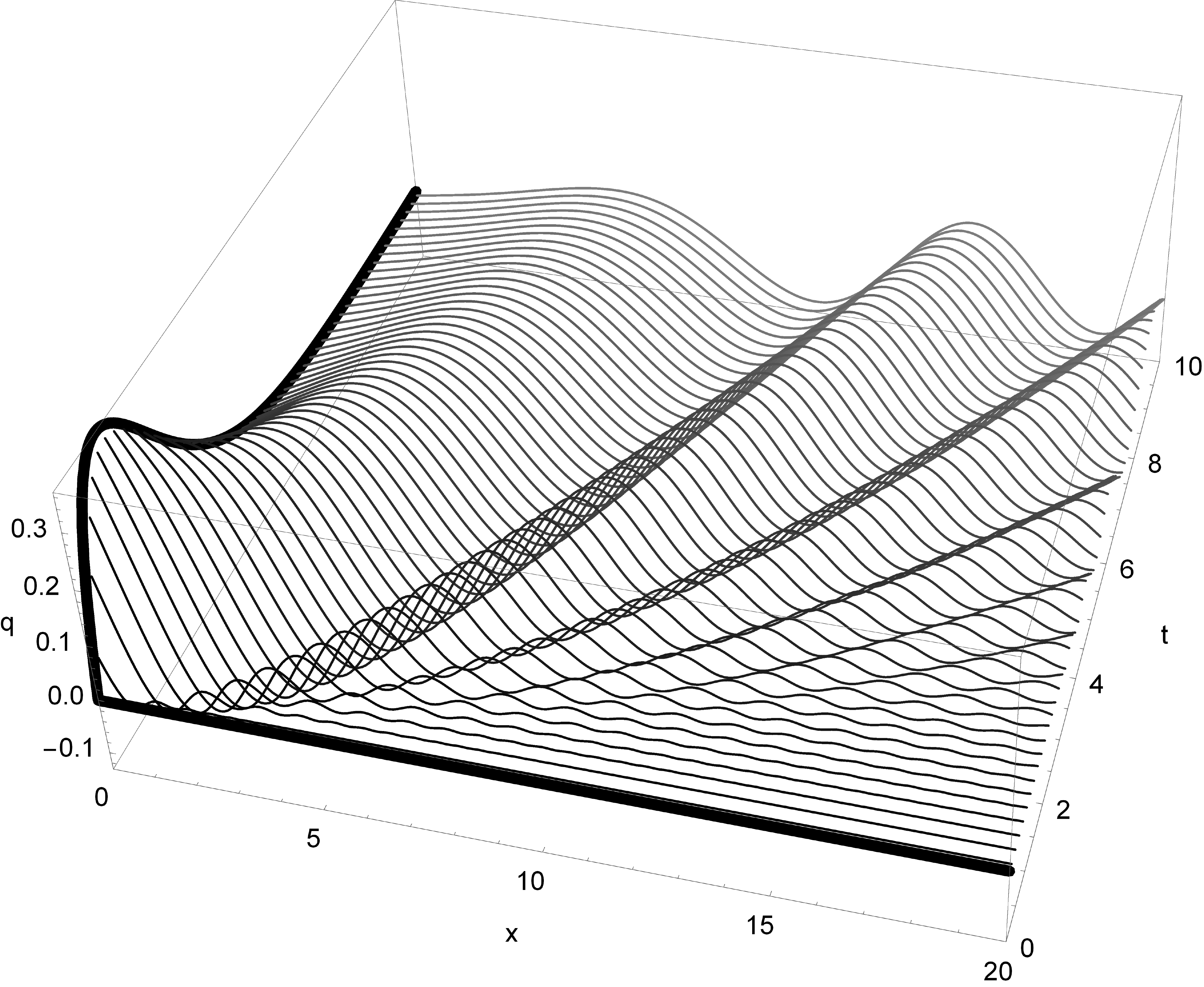}
   }
  \caption{The plot of the real part of the solution of the LS equation with $q_0(x)=0$, $g_0(t)=te^{-t}$. The bold solid curve shows the initial and boundary conditions.}
  \label{lssol}
\end{figure}

\begin{figure}
  \makebox[\textwidth][c]{
 
 \begin{subfigure}{0.3\textwidth} 
  \includegraphics[width=1\textwidth]{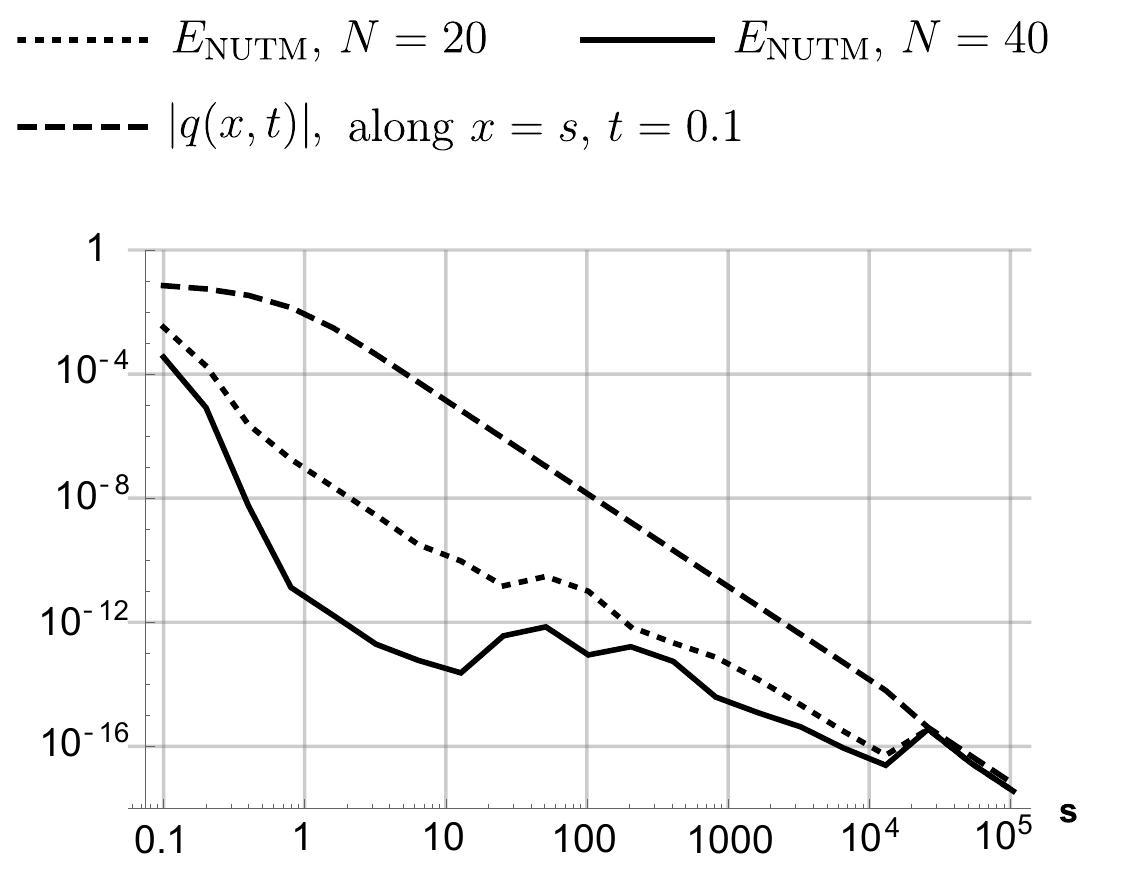}
      \caption{}
    \end{subfigure}

 \begin{subfigure}{0.3\textwidth} 
  \includegraphics[width=1\textwidth]{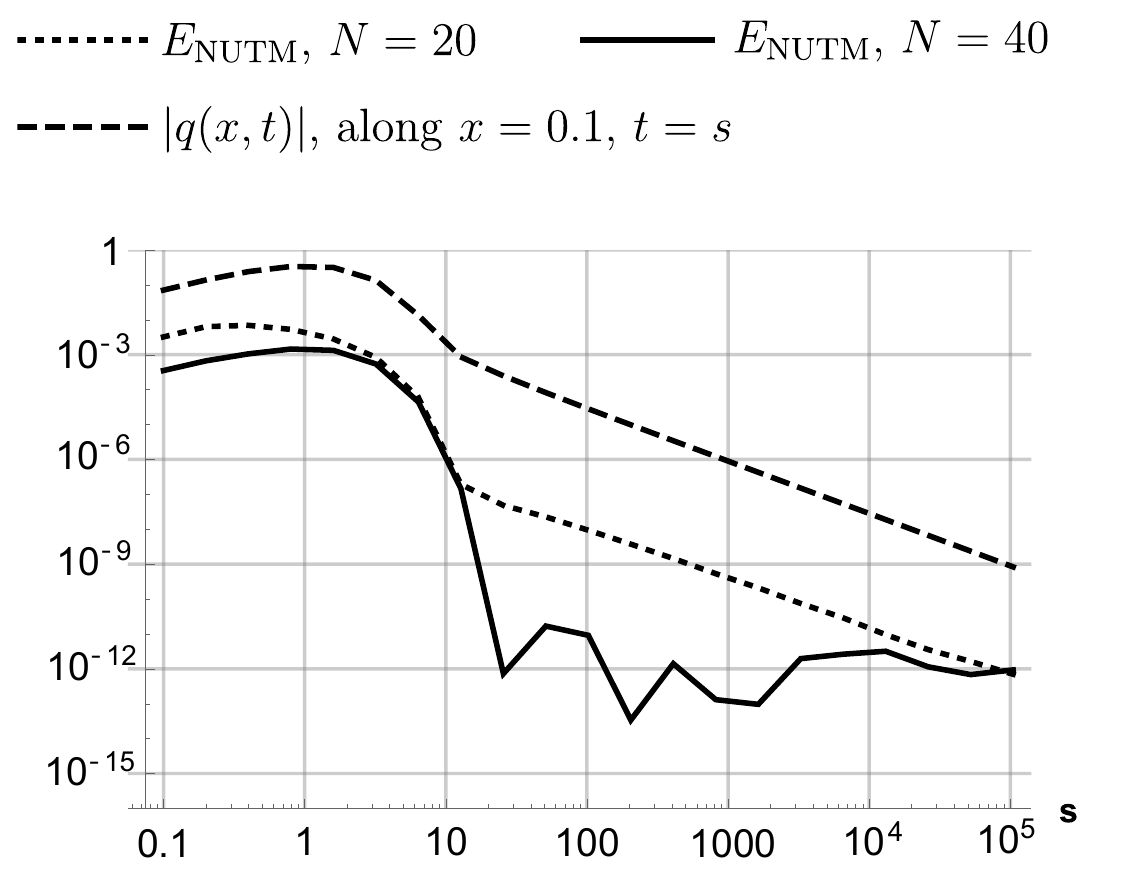}
      \caption{}
    \end{subfigure}

 \begin{subfigure}{0.3\textwidth} 
  \includegraphics[width=1\textwidth]{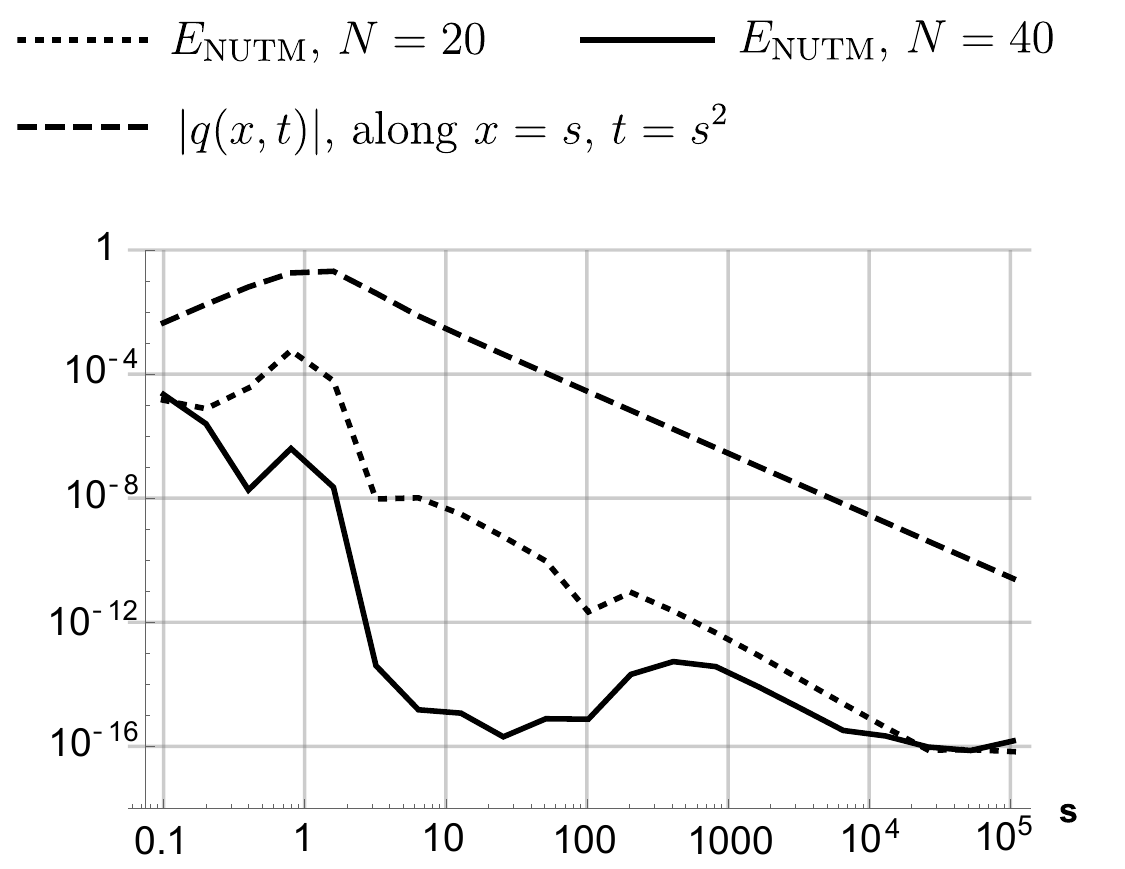}
     \caption{}
    \end{subfigure}

  }
  \caption{The absolute errors $E_{\text{NUTM}}$ of the numerical solution to the LS equation (\ref{lseq}) along three curves: (a) $x=s,t=0.1$, (b) $x=0.1, t=s$, (c) $x=s,t=s^2$  for $s\in [0.1,10^5]$.}
\label{lserr}
\end{figure}

\section{A multi-term third-order PDE}
\label{ch_lkdv}
The deformations for higher-order equations are more involved and the integrands may have branch points that are fixed by the equation and not by the initial or boundary data. The NUTM is implemented in a systematic way as long as one can solve the PDE using the UTM with additional care for the branch points. Consider a multi-term third-order PDE, 
\begin{align}
q_t=q_x+q_{xxx},~~~~x>0,~t>0,
\label{lkdveq}
\end{align}
with Dirichlet boundary data $g_0 \in C^{\infty}_\gamma$, Neumann boundary data $g_1 \in C^{\infty}_\gamma$ and initial data $q_0 \in C^{\infty}_\delta$. 
The dispersion relation is $\omega(k)=-ik+ik^3$ and $D^+=\{k\in \mathbb C^+:\re{\omega(k)}<0\}=D^+_1\cup D^+_2$ as shown in Figure~\ref{lkdvregion}. 
\begin{figure}
  \makebox[\textwidth][c]{
  \includegraphics[width=0.8\textwidth]{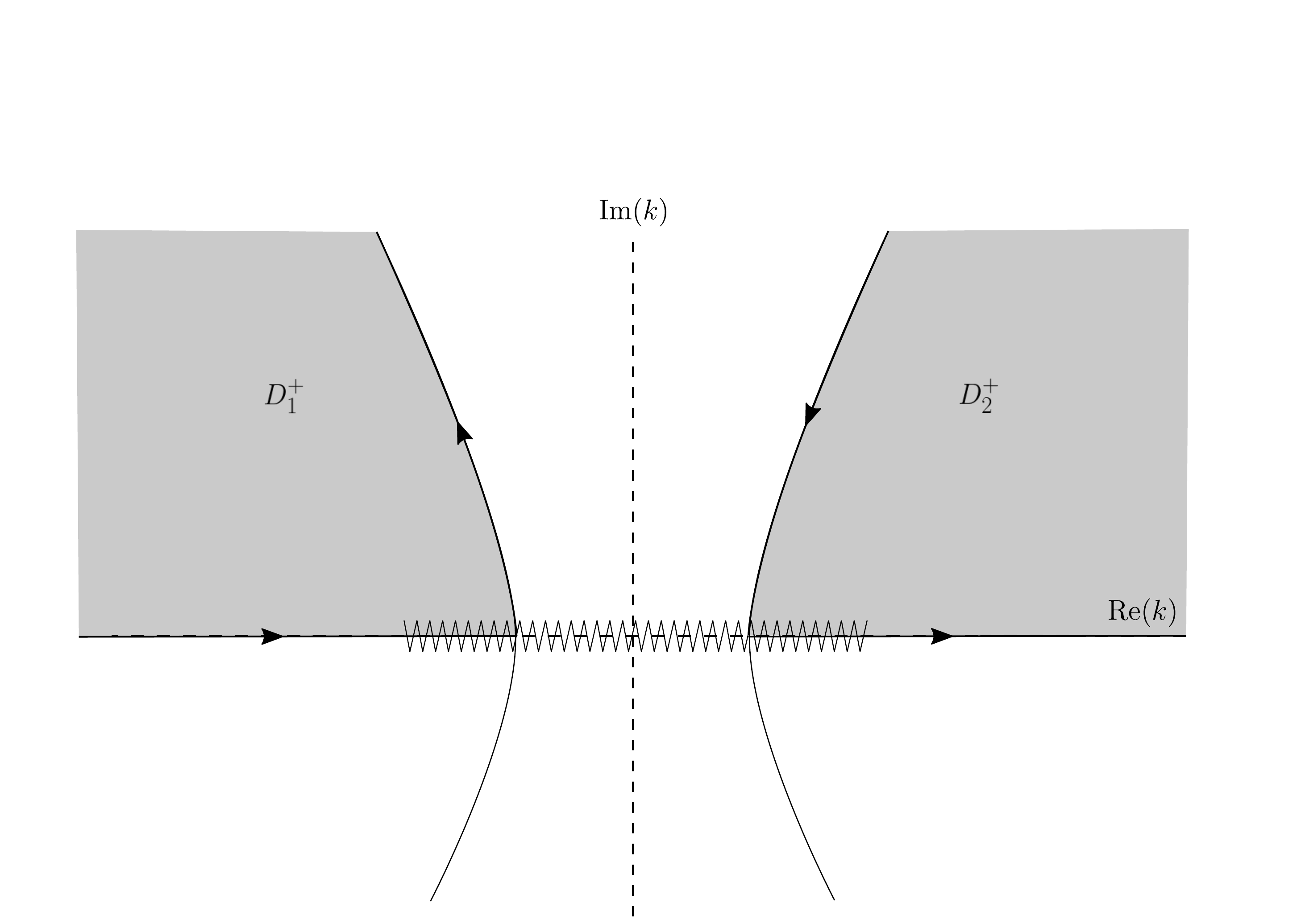}
  }
  \caption{The region $D^+$ for (\ref{lkdveq}). The shaded region in the top right is $D^+_1$. The shaded region in the top left is $D^+_2$. The branch cut is shown as a jagged line.}
  \label{lkdvregion}
\end{figure}
\begin{figure}
  \makebox[\textwidth][c]{
  \includegraphics[width=0.8\textwidth]{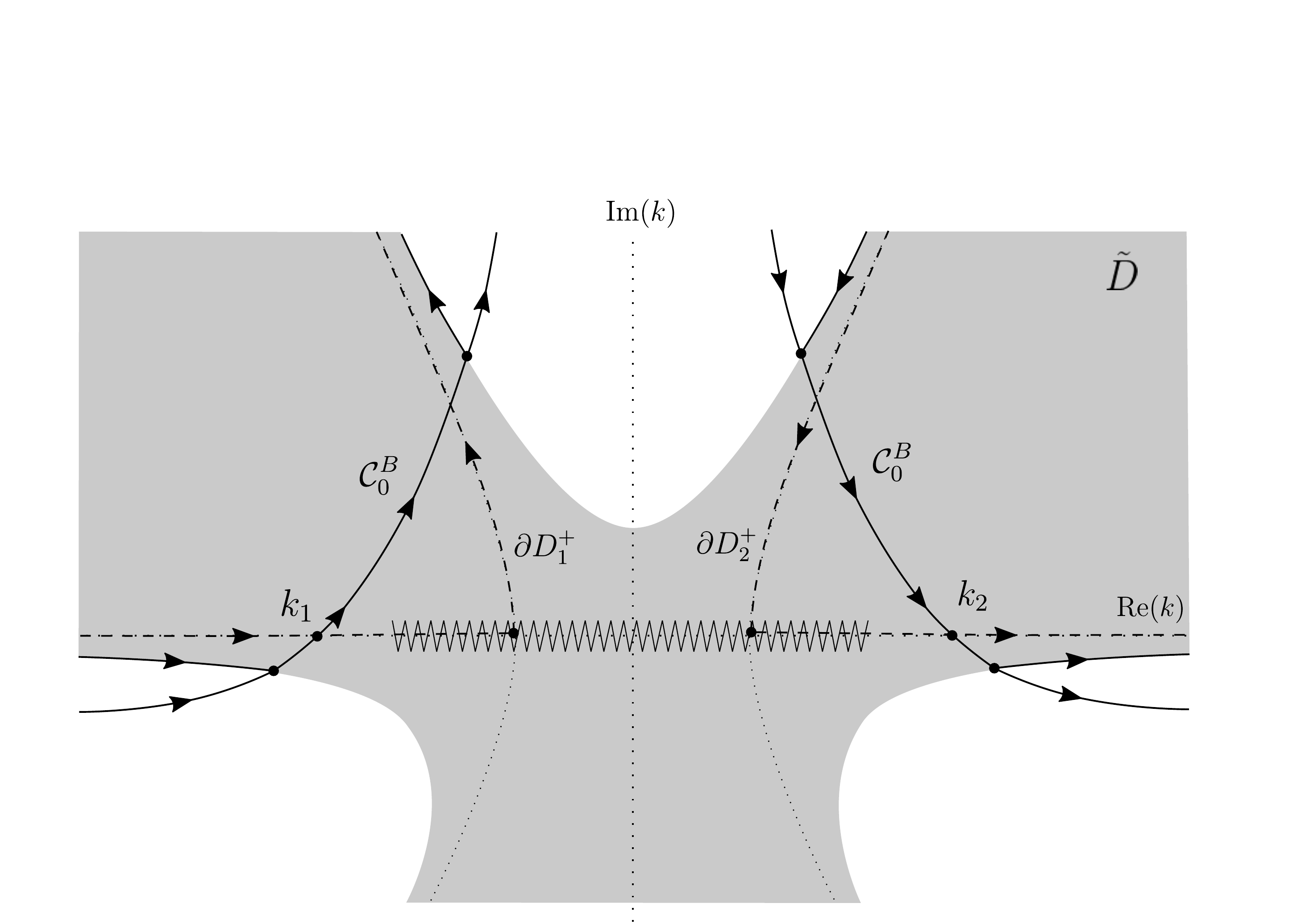}
  }
  \caption{The region $\tilde{D}$ and the deformation for $B_0$ across the saddle points $k_1$ and $k_2$. The branch cut is shown as a jagged line.}
  \label{lkdvregion2}
\end{figure}

Using the UTM, it is known that the problem requires two boundary conditions at $x=0$ \cite{deconinck2014}. 
By solving $\omega(\nu(k))=\omega(k)$,  we find two symmetries of the dispersion relation,
\begin{align}
\nu_1(k) =  (-k - \sqrt{4 - 3 k^2})/2,\\
\nu_2(k) =  (-k + \sqrt{4 - 3 k^2})/2,
\end{align}
with branch cut $[-2/  \sqrt{3}, 2/\sqrt{3}]$. Here, $\nu_1$ is the branch of $\nu$ that tends to $(-1/2+i\sqrt{3}/2)k=k\exp(2\pi i /3)$ as $k\rightarrow \infty$ and $\nu_2$ is the other branch. The solution formula is given by\footnote{Although some parts of the contours lie on branch cut, the integrands are well-defined if the values are taken as limits from the interior of  $D^+$.}
\begin{align}
q(x,t)=I_1+I_2+I_3+B_0+B_1,
\label{sollkdv}
\end{align}
with
\begin{align}
I_1&=\frac{1}{2\pi} \int_{-\infty}^{\infty} e^{ikx-\omega(k)t} \hat{q}_0(k)dk,\\
I_2&=-\frac{1}{2\pi} \int_{\partial D_1^+}e^{ikx-\omega(k)t} \hat{q}_0(\nu_1(k))  dk,\\
I_3&=-\frac{1}{2\pi} \int_{\partial D_2^+}e^{ikx-\omega(k)t} \hat{q}_0(\nu_2(k))  dk,\\
B_0&=-\frac{1}{2\pi}  \int_{\partial D_1^+} e^{ikx-\omega(k)t} (\nu^2_1(k)-k^2)\tilde{g}_0(\omega(k),t)dk -\frac{1}{2\pi}  \int_{\partial D_2^+} e^{ikx-\omega(k)t}  (\nu^2_2(k)-k^2)\tilde{g}_0(\omega(k),t)dk, \\
B_1&=-\frac{1}{2\pi}  \int_{\partial D_1^+} e^{ikx-\omega(k)t}(ik-i\nu_1(k))\tilde{g}_1(\omega(k),t) dk -\frac{1}{2\pi}  \int_{\partial D_2^+} e^{ikx-\omega(k)t}(ik-i\nu_2(k))\tilde{g}_1(\omega(k),t) dk.
\end{align}
For convenience, we impose the following initial and boundary conditions to focus on the deformation of $B_0$,  
\[
q(x,0)=0,\,\, q(0,t)=g_0(t),\,\, g_0 \in C_{\gamma}^\infty,\,\, q_x(0,t)=0.
\]  
For inhomogeneous initial and Neumann data, the deformation of $B_1$ follows the same steps as the deformation of $B_0$ and the deformations of $I_1,I_2,I_3$ follow the same steps as in $I_1,I_2$ in the heat equation or the LS equation case.  
\subsection{Deformations of the contour of $B_0$ based on the method of steepest descent}
With homogeneous initial and Neumann boundary conditions, the solution reduces to
\begin{align}
q(x,t)=B_0=B_0|_{D_1^+}+B_0|_{D_2^+},
\end{align}
where 
\begin{align}
B_0|_{D_1^+}&= -\frac{1}{2\pi}  \int_{\partial D_1^+} e^{ikx-\omega(k)t} (\nu_1^2(k)-k^2)\tilde{g}_0(\omega(k),t)dk,\\
B_0|_{D_2^+}&= -\frac{1}{2\pi}  \int_{\partial D_2^+} e^{ikx-\omega(k)t}  (\nu^2_2(k)-k^2)\tilde{g}_0(\omega(k),t)dk.
\end{align}
The phase function in $B_0$ is
\begin{align}
\theta(k;x,t)=ikx-\omega(k)t=ikx-(-ik+ik^3)t.
\end{align}
There are two saddle points $k_{1,2}=\pm \sqrt{x/(3t)+1/3}$ on the real axis satisfying $\theta'(k;x,t)=0$, $k_1\in D_1^+$ and $k_2\in D_2^+$. Since the saddle points and contours are symmetric with respect to the imaginary axis, we only need to analyze the deformation for $D_2^+$ and use the mirror image about the imaginary axis for $D_1^+$. Near the saddle point $k_2$, $\theta$ has the expansion
\[
\theta(k;x,t)= \frac{2}{9}i  \left( t\sqrt{\frac{3(t + x)}{t}} +  x \sqrt{\frac{3(t + x)}{t} } \right) - 
 i  t \sqrt{\frac{3(t + x)}{t}} (k - k_2)^2 + \mathcal{O}(k-k_2)^3.
 \]
The direction of steepest descent is along the angles $-\pi/4$ and $3\pi/4$. The integrands need to be extended to the lower half plane similar to the steps in Section \ref{sec_lsi0}.

\subsubsection{Deformations of the contour of $B_0$ for $x\geq 3t$}
In this case the saddle points $k_1$,$k_2$ lie outside branch cut $[-2/\sqrt{3}, 2\sqrt{3}]$. We proceed as follows.
\begin{enumerate}
\item The region $D^+=\{k\in \mathbb{C}^+:\re{\omega(k)}<0\}$ is extended to $\tilde{D}=\{k\in \mathbb{C}: \re{\omega(k)}<\gamma\}$. The transformed data $\tilde{g}_0(\omega(k), \infty)$ is analytic in $\tilde{D}$, and continuous up to  $\partial \tilde{D}$.
\item The contour $\partial D^+$ is deformed to $\mathcal{C}_{0,a}^{B} \cup \mathcal{C}_{0,b}^{B} \cup \mathcal{C}_{0,c}^{B}$ as shown in Figure \ref{lkdvctp1}. $ \mathcal{C}_{0,b}^{B}$ is the curve passing through the saddle point up to $\partial \tilde{D}$, keeping $\im{\theta(k;x,t)}$ constant along the steepest-descent direction and $\mathcal{C}_{0,a}^{B}, \mathcal{C}_{0,c}^{B}$ are the unbounded curve segments along $\partial \tilde{D}$.
\item Using that $e^{ikx-\omega(k)t} \int_t^{\infty}e^{\omega(k) s}g_0(s)ds$ is bounded and analytic in $\tilde{D}$, we can replace $\tilde{g}_0(\omega(k),t)$ with $\tilde{g}_0(\omega(k),\infty)$,
\begin{align*}
 B_0|_{D_2^+}=&\frac{1}{2\pi}\int_{\partial D^+_2} e^{ikx-\omega(k)t}  (\nu^2_2(k)-k^2)\tilde{g}_0(\omega(k),t)dk \\
 =&\frac{1}{2\pi}\int_{\mathcal{C}_{0,a}^{B}\cup \mathcal{C}_{0,b}^{B}\cup \mathcal{C}_{0,c}^{B}} e^{ikx-\omega(k)t}  (\nu^2_2(k)-k^2)\tilde{g}_0(\omega(k),\infty) dk .
 \label{i0ginf}
 \end{align*} 
\item The integral along $\mathcal{C}_{0,a}^{B}$ is decomposed into two parts to maximize decay along the steepest-descent direction,
 \begin{align*}
\int_{\mathcal{C}_{0,a}^{B}}e^{ikx-\omega(k)t}  (\nu^2_2(k)-k^2)\tilde{g}_0(\omega(k),\infty) dk =&\\
\int_{\mathcal{C}_{0,a}^{B}}e^{ikx-\omega(k)t}  (\nu^2_2(k)-k^2)\tilde{g}_0^c(\omega(k),t) dk&+\int_{\mathcal{C}_{0,d}^{B}}e^{ikx-\omega(k)t} (\nu^2_2(k)-k^2)\tilde{g}_0(\omega(k),t)dk , 
 \end{align*} 
 where 
\[
\tilde{g}^c_0(\omega(k),t)=\int_t^{\infty} e^{\omega(k)s} g_0(s)ds ,
\]
is the complement of the transform $\tilde{g}_0(\omega(k),t)$ and $\mathcal{C}_{0,d}^{B}$ is extended from $\mathcal{C}_{0,b}^{B}$ keeping $\im{\theta(k;x,t)}$ constant along the steepest-descent direction.
\item The integral along $\mathcal{C}_{0,c}^{B}$ is decomposed into two parts to separate the leading-order contribution in the oscillatory integral,
 \begin{align*}
\int_{\mathcal{C}_{0,c}^{B}}e^{ikx-\omega(k)t} (\nu^2_2(k)-k^2)\tilde{g}_0(\omega(k),\infty) dk=\int_{\mathcal{C}_{0,e}^{I}}e^{ikx-\omega(k)t}  (\nu^2_2(k_2)-k_2^2)\tilde{g}_0(\omega(k_2),\infty)dk\\ 
+\int_{\mathcal{C}_{0,c}^{B}}e^{ikx-\omega(k)t}\left[ (\nu^2_2(k)-k^2)\tilde{g}_0(\omega(k),\infty)- (\nu^2_2(k_2)-k_2^2)\tilde{g}_0(\omega(k_2),\infty) \right]dk.
\end{align*} 
The contour $\mathcal{C}_{0,e}^{B}$ is extended from $\mathcal{C}_{0,b}^{B}$ keeping $\im{\theta(k;x,t)}$ constant along the steepest-descent direction.
\item Finally, we obtain
\begin{align*}
B_0|_{D_2^+}=&\frac{1}{2\pi}\int_{\mathcal{C}_{0,a}^{I}}e^{ikx-\omega(k)t} (\nu^2_2(k)-k^2)\tilde{g}^c_0(\omega(k),t) dk+\frac{1}{2\pi}\int_{\mathcal{C}_{0,d}^{B}}e^{ikx-\omega(k)t} (\nu^2_2(k)-k^2)\tilde{g}_0(\omega(k),t)dk\\
 &+\frac{1}{2\pi}\int_{\mathcal{C}_{0,b}^{B}}e^{ikx-\omega(k)t} (\nu^2_2(k)-k^2)\tilde{g}_0(\omega(k),\infty) dk \\
 &+\frac{1}{2\pi}\int_{\mathcal{C}_{0,c}^{B}}e^{ikx-\omega(k)t}(\nu^2_2(k)-k^2)\tilde{g}_0(\omega(k),\infty) - (\nu^2_2(k_2)-k_2^2)\tilde{g}_0(\omega(k_2),\infty)dk\\
 &+\frac{1}{2\pi}\int_{\mathcal{C}_{0,e}^{B}}e^{ikx-\omega(k)t}(\nu^2_2(k_2)-k_2^2)\tilde{g}_0(\omega(k_2),\infty)dk.
\end{align*}
The integrals along $\mathcal{C}_{0,b}^{B},\mathcal{C}_{0,d}^{B}$ and $\mathcal{C}_{0,e}^{B}$ are computed using Clenshaw-Curtis quadrature and the integrals along $\mathcal{C}_{0,a}^{B}$ and $\mathcal{C}_{0,c}^{B}$ are computed using Levin's method. 
\end{enumerate}
The contour integral $B_0|_{D_1^+}$ is deformed in a symmetrical way. For real-valued data, we can use the symmetry and compute $q(x,t)$ with only the contour integral $B_0|_{D_2^+}$,
\[
q(x,t)=2\re{B_0|_{D_2^+}}.
\]
\begin{figure}
  \makebox[\textwidth][c]{
  \includegraphics[width=0.5\textwidth]{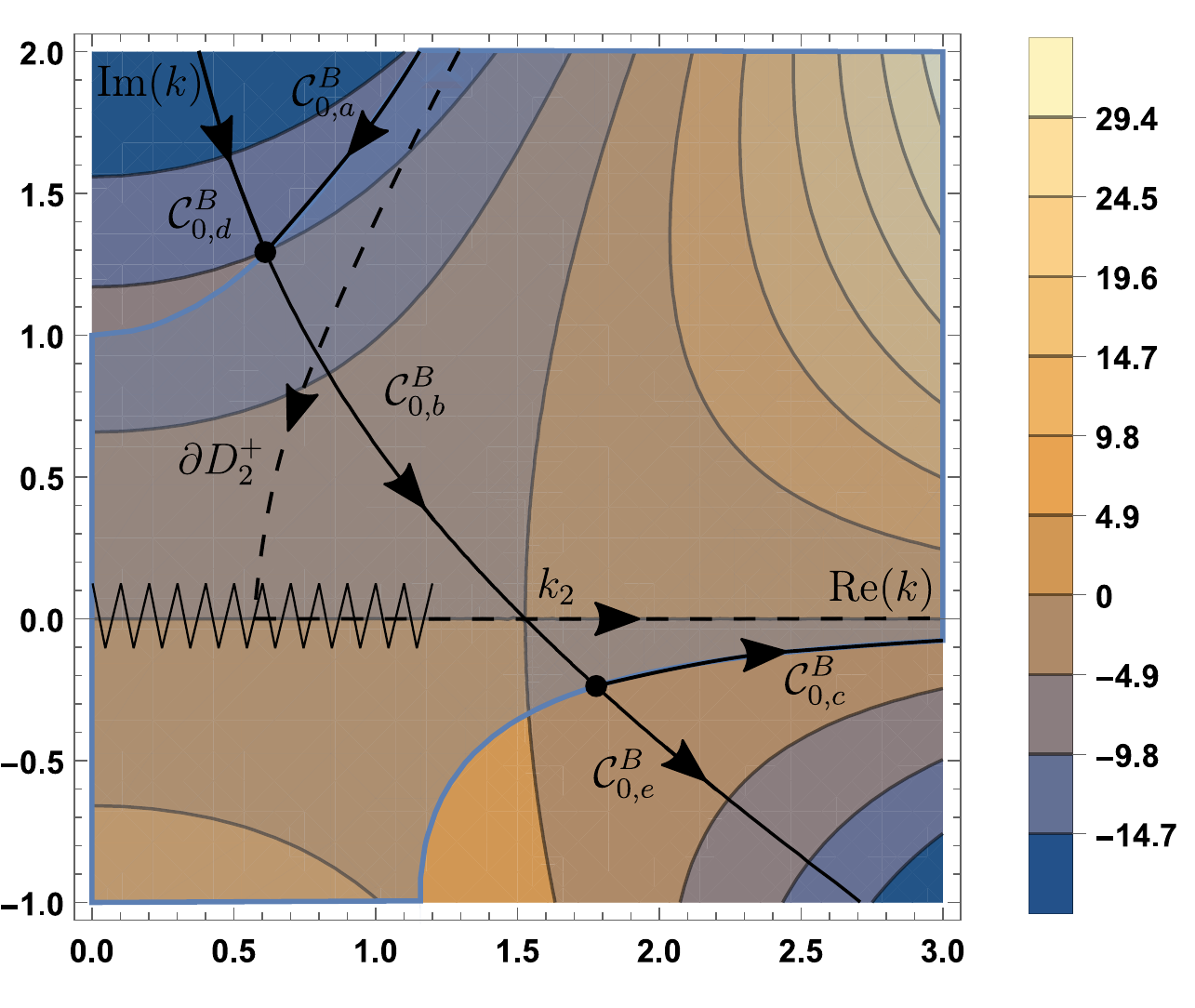}
  }
  \caption{The deformed contour for $B_0|_{D_2}$ when $x\geq 3t$ (solid), the undeformed contour (dashed). The background contour  plot shows the level sets of  $\mbox{Re}(\theta(k,x,t))$. The branch cut is shown as a jagged line.}
  \label{lkdvctp1}
\end{figure}

\subsubsection{Deformations of the contour for $B_0|_{D_2^+}$ for $x< 3t$}
When $x< 3t$, the saddle points $k_1$, $k_2$ lie on branch cut $[-2/\sqrt{3}, 2/\sqrt{3}]$. To maximize the use of the steepest-decent direction, we choose a different branch cut for $\nu$, shown in Figure \ref{lkdvctp2} in red. The new branch cut starts from the branch point $2/\sqrt{3}$ and goes along the curve with $\im{\theta(k;x,t)}$ constant in the lower half plane. The corresponding $\tilde{\nu}_2(k)$ is defined as the analytic continuation of $\nu_2(k)$ from the interior of $D_2^+$.  We use the following steps.
\begin{enumerate}
\item The region $D^+=\{k\in \mathbb{C}^+:\re{\omega(k)}<0\}$ is extended to $\tilde{D}=\{k\in \mathbb{C}: \re{\omega(k)}<\gamma\}$. The transformed data $\tilde{g}_0(\omega(k), \infty)$ is analytic in $\tilde{D}$, and continuous up to  $\partial \tilde{D}$.
\item The contour $\partial D_2^+$ is deformed to $\mathcal{C}_{0,a}^{B} \cup \mathcal{C}_{0,b}^{B} \cup \mathcal{C}_{0,e}^{B} \cup (-\mathcal{C}_{0,f}^{B}) \cup \mathcal{C}_{0,f}^{B}   \cup \mathcal{C}_{0,c}^{B}$ as shown in Figure \ref{lkdvctp2}. The contour $ \mathcal{C}_{0,b}^{B}$ is the curve passing through the saddle point up to $\partial \tilde{D}$ with $\im{\theta(k;x,t)}$ constant along the steepest-descent direction. The contours $\mathcal{C}_{0,a}^{B}, \mathcal{C}_{0,e}^{B}$ and $ \mathcal{C}_{0,c}^{B}$ are along $\partial \tilde{D}$. The contours $-\mathcal{C}_{0,f}^{B}$ and $\mathcal{C}_{0,f}^{B}$ are the two segments on the new branch cut with opposite orientations. The contour $-\mathcal{C}_{0,f}^{B}$ points towards the branch point and $\mathcal{C}_{0,f}^{B}$ points away from the branch point.
\item Using that $e^{ikx-\omega(k)t} \int_t^{\infty}e^{\omega(k) s}g_0(s)ds$ is bounded and analytic in $\tilde{D} $, we can replace $\tilde{g}_0(\omega(k),t)$ with $\tilde{g}_0(\omega(k),\infty)$,
\begin{align*}
B_0|_{D_2^+}=&\frac{1}{2\pi}\int_{\partial D^+_2} e^{ikx-\omega(k)t}  (\tilde{\nu}^2_2(k)-k^2)\tilde{g}_0(\omega(k),t)dk \\
 =&\frac{1}{2\pi}\int_{\mathcal{C}_{0,a}^{B}\cup \mathcal{C}_{0,b}^{B}\cup \mathcal{C}_{0,e}^{B}\cup \mathcal{C}_{0,c}^{B}} e^{ikx-\omega(k)t}  (\tilde{\nu}^2_2(k)-k^2)\tilde{g}_0(\omega(k),\infty) dk \\
  &-  \frac{1}{2\pi}\int_{\mathcal{C}_{0,f}^{B}} e^{ikx-\omega(k)t}  \tilde{\nu}^2_2(k^-)\tilde{g}_0(\omega(k),\infty) dk+  \frac{1}{2\pi}\int_{\mathcal{C}_{0,f}^{B}} e^{ikx-\omega(k)t}  \tilde{\nu}^2_2(k^+)\tilde{g}_0(\omega(k),\infty) dk, 
 \label{i0ginf}
\end{align*} 
where $k^+$ and $k^-$ denote the limit from the left/right of the curve respectively.
\item The integral along $\mathcal{C}_{0,a}^{B}$ is decomposed into two parts to maximize decay along the steepest-descent direction: \begin{align*}
\int_{\mathcal{C}_{0,a}^{B}}e^{ikx-\omega(k)t}  (\tilde{\nu}^2_2(k)-k^2)\tilde{g}_0(\omega(k),\infty) dk =&\\
\int_{\mathcal{C}_{0,a}^{B}}e^{ikx-\omega(k)t}  (\tilde{\nu}^2_2(k)-k^2)\tilde{g}_0^c(\omega(k),t) dk&+\int_{\mathcal{C}_{0,d}^{B}}e^{ikx-\omega(k)t} (\tilde{\nu}^2_2(k)-k^2)\tilde{g}_0(\omega(k),t)dk  ,
 \end{align*} 
\item We obtain
\begin{align*}
 B_0|_{D_2^+}=&\frac{1}{2\pi}\int_{\mathcal{C}_{0,a}^{B}}e^{ikx-\omega(k)t} (\tilde{\nu}^2_2(k)-k^2)\tilde{g}^c_0(\omega(k),t) dk+\frac{1}{2\pi}\int_{\mathcal{C}_{0,d}^{B}}e^{ikx-\omega(k)t} (\tilde{\nu}^2_2(k)-k^2)\tilde{g}_0(\omega(k),t)dk\\
 &+\frac{1}{2\pi}\int_{\mathcal{C}_{0,b}^{B}}e^{ikx-\omega(k)t} (\tilde{\nu}^2_2(k)-k^2)\tilde{g}_0(\omega(k),\infty) dk \\
& +\frac{1}{2\pi}\int_{\mathcal{C}_{0,e}^{B} \cup \mathcal{C}_{0,c}^{B}}e^{ikx-\omega(k)t}(\tilde{\nu}^2_2(k)-k^2)\tilde{g}_0(\omega(k),\infty) dk\\
&+  \frac{1}{2\pi}\int_{-\mathcal{C}_{0,f}^{B}} e^{ikx-\omega(k)t}  \tilde{\nu}^2_2(k^+)\tilde{g}_0(\omega(k),\infty) dk+  \frac{1}{2\pi}\int_{\mathcal{C}_{0,f}^{B}} e^{ikx-\omega(k)t}  \tilde{\nu}^2_2(k^+)\tilde{g}_0(\omega(k),\infty) dk.
\end{align*} 
The integrals along $\mathcal{C}_{0,b}^{B},\mathcal{C}_{0,d}^{B},-\mathcal{C}_{0,f}^{B}$ and $\mathcal{C}_{0,f}^{B}$ are computed using Clenshaw-Curtis quadrature and the integrals along $\mathcal{C}_{0,a}^{B},\mathcal{C}_{0,c}^{B}$ and $\mathcal{C}_{0,e}^{B}$ are computed using Levin's method. The contour integral $B_0|_{D_1^+}$ is deformed symmetrically.
\end{enumerate}
\begin{figure}
  \begin{subfigure}[c]{.5\linewidth}
    \includegraphics[width=1.0\textwidth]{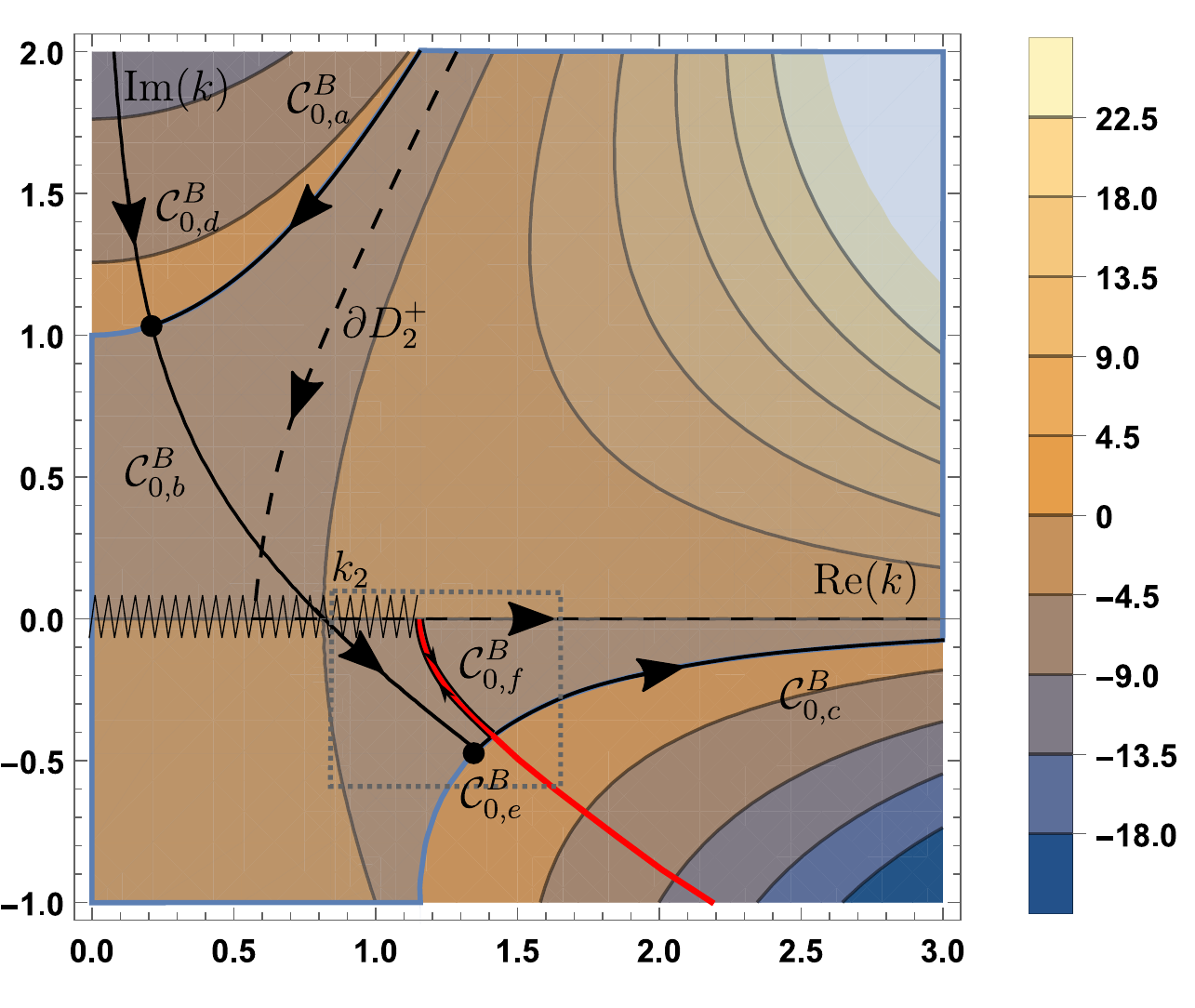}
  \end{subfigure}
  \begin{subfigure}[c]{.5\linewidth}
  \includegraphics[width=0.8\textwidth]{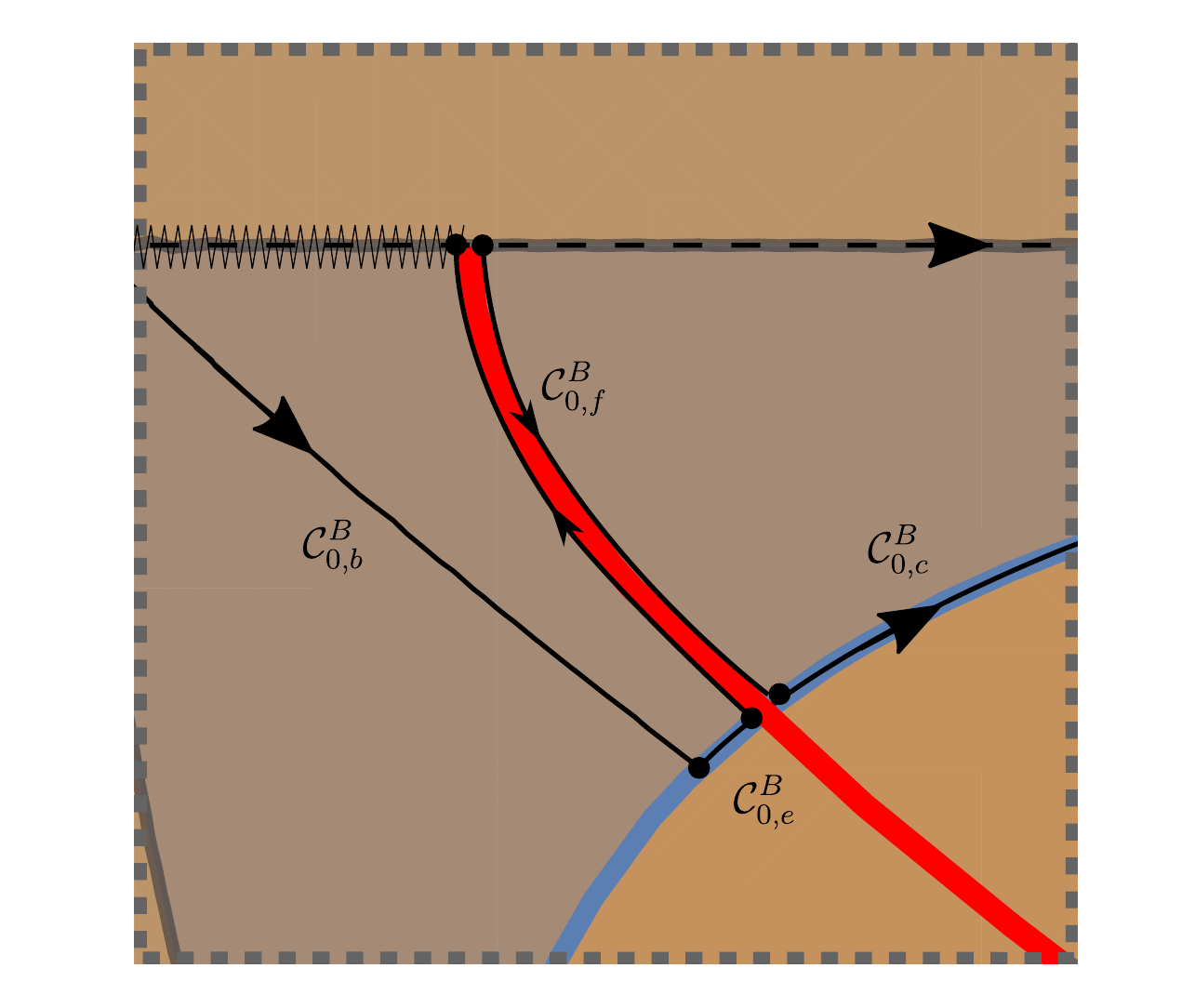}
  \end{subfigure}
  \caption{The deformed contour for $B_0|_{D_2}$ when $x< 3t$ (solid). The undeformed contour (dashed).  The background contour  plot shows the level sets of $\mbox{Re}(\theta(k,x,t))$. The original branch cut is shown as a jagged line and the new branch cut is shown in red. A zoomed plot of the contour near the new branch cut is shown in the right panel.}
  \label{lkdvctp2}
\end{figure}

\subsubsection{Improve the accuracy near the branch point}
\label{sec_cov}
Since $\tilde{\nu}_2(k)$ is not differentiable at the branch point $k_2=2/\sqrt{3}$, Clenshaw-Curtis quadrature loses spectral accuracy for the integrals along $-\mathcal{C}_{0,f}^{B}$, $\mathcal{C}_{0,f}^{B}$ and $\mathcal{C}_{0,b}^{B}$ in the critical case $x=3t$. With the change of variables $s^2=k-k_2$, we get
\[
\hat{\nu}_2(s):=\nu_2(s^2+k_2)= (-2 \sqrt{3} - 3 s^2 -  i 3^{5/4} s \sqrt{4 + \sqrt{3} s^2})/6.
\]
The new symmetry $\hat{\nu}_2(s)$ is smooth near $s=0$. Clenshaw-Curtis quadrature maintains spectral accuracy for the integrals on $\mathcal{C}_{0,b}^{B}$ and $\mathcal{C}_{0,f}^{B}$ after this change of variables.

\subsection{Numerical examples}
Consider the Dirichlet boundary condition $g_0(t)=te^{-t}$, the homogeneous initial condition $q_0(x)=0$ and the Neumann boundary condition $g_1(t)=0$. The solution to (\ref{lkdveq}) is shown in Figure \ref{lkdvsol}. For small time, the dispersive waves emanate from the boundary and the solution looks similar to Figure \ref{lssol}. As $t$ grows, the advection dominates and the waves turn back to the boundary. The absolute error and the magnitude of the solution evaluated along (a) $t=0.1$, (b) $x=0.1$, (c) $x=3t$ are shown in Figure \ref{lkdverr}. The errors shown in dotted curves are computed with $N=20$ collocation points for each part of the contour in $B_0$ while the errors shown in solid curves are computed with $N=40$ collocation points. The absolute errors tend to zero as $x,t$ increase. 
To demonstrate spectral accuracy, the absolute errors $E_{\text{NUTM}}$ evaluated at $x=1,3,5$, $t=1$ are plotted against the number of collocation points per segment in Figure \ref{lkdverr2}. With the change of variables used in Section \ref{sec_cov}, the NUTM remains spectrally accurate even when the branch point is on the contour of integration.

All our examples use boundary conditions with transforms that can be computed explicitly. This is to allow us to estimate the error of our method by comparing with the built-in integration routine in Mathematica. To show the NUTM is not limited to this, in Figure \ref{lkdvsol2}, we show a plot of the solution $q(x,t)$ with $g_0(t)=\sin(2t) \phi(t/(2\pi))$ where
\begin{align}
\phi(t)= \begin{cases} 
      \exp(-1/(1 - t^2)) & \abs{t} \leq 1, \\
      0 & \abs{t}>1.
   \end{cases}
   \label{molifier}
\end{align}
The initial data and the Neumann data are zero. We see a similar wave pattern as in Figure \ref{lkdvsol} with dispersive waves propagating in the positive $x$ direction, before turning back.

\begin{figure}
  \makebox[\textwidth][c]{
  \includegraphics[width=0.9\textwidth]{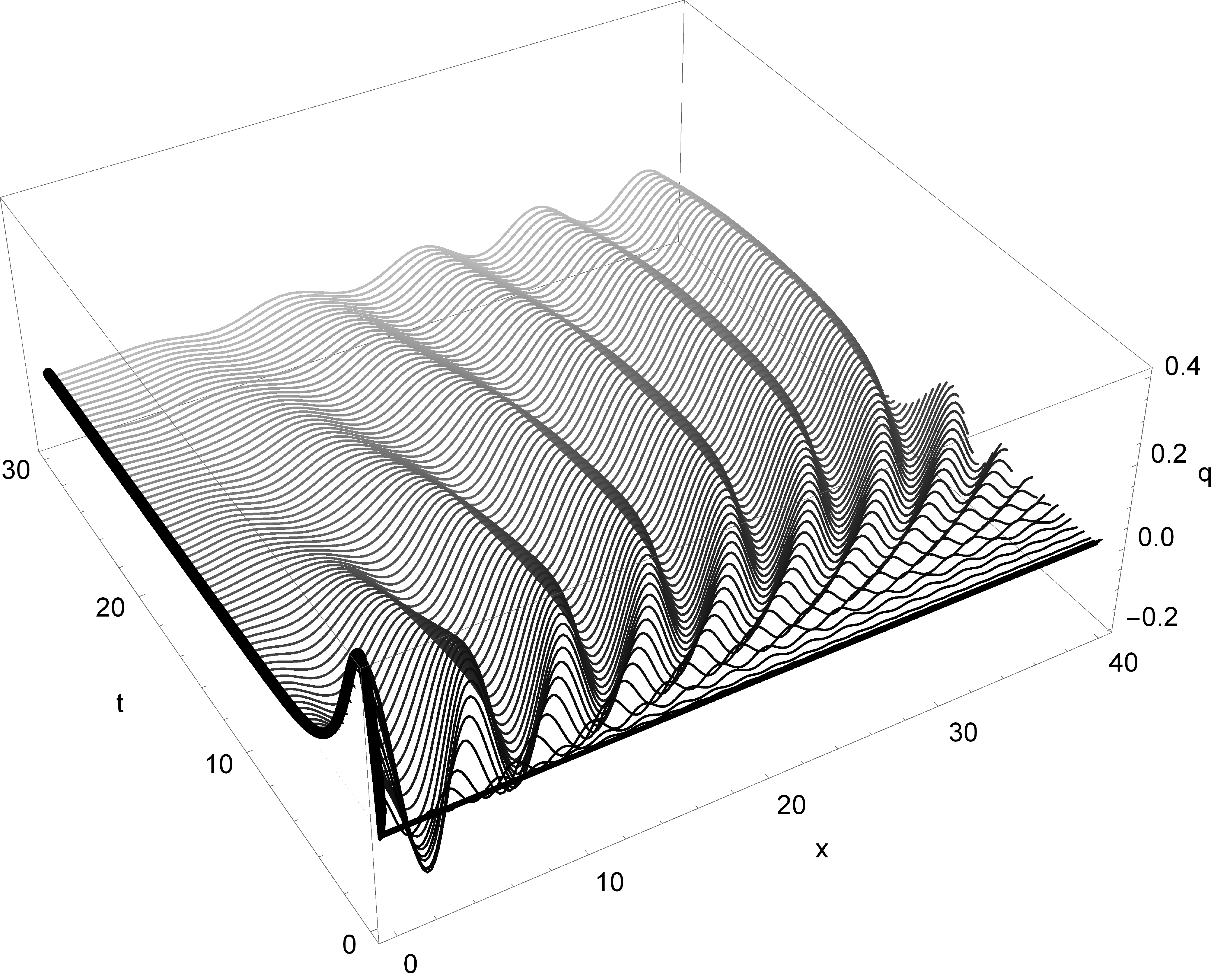}
  }
  \caption{The numerical solution of equation (\ref{lkdveq}) with $q_0(x)=0$, $g_0(t)=te^{-t}$, $g_1(t)=0$. The bold curves are the initial and Dirichlet boundary conditions. For small $t$, dispersive waves emanate from the boundary while the waves start to turn back following the advection as $t$ grows.}
  \label{lkdvsol}
\end{figure}

\begin{figure}
  \makebox[\textwidth][c]{

 \begin{subfigure}{0.3\textwidth} 
  \includegraphics[width=\textwidth]{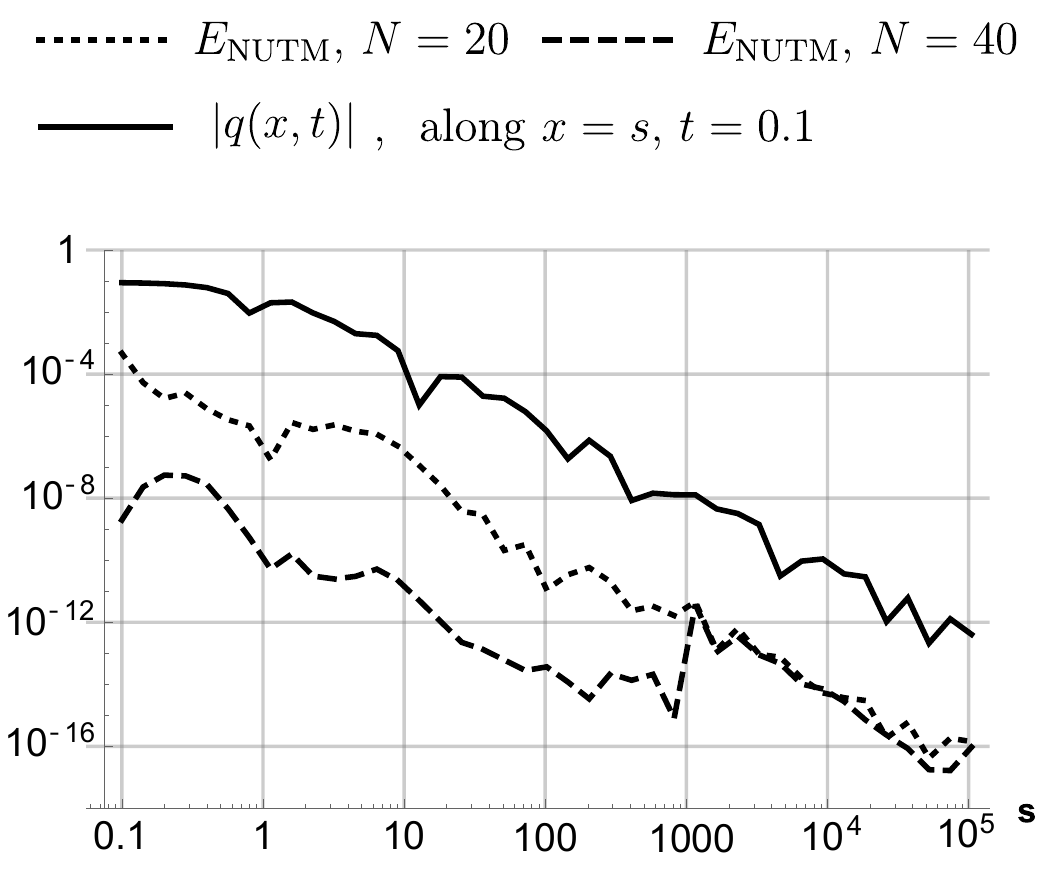}
\caption{}
\end{subfigure}

 \begin{subfigure}{0.3\textwidth} 
  \includegraphics[width=\textwidth]{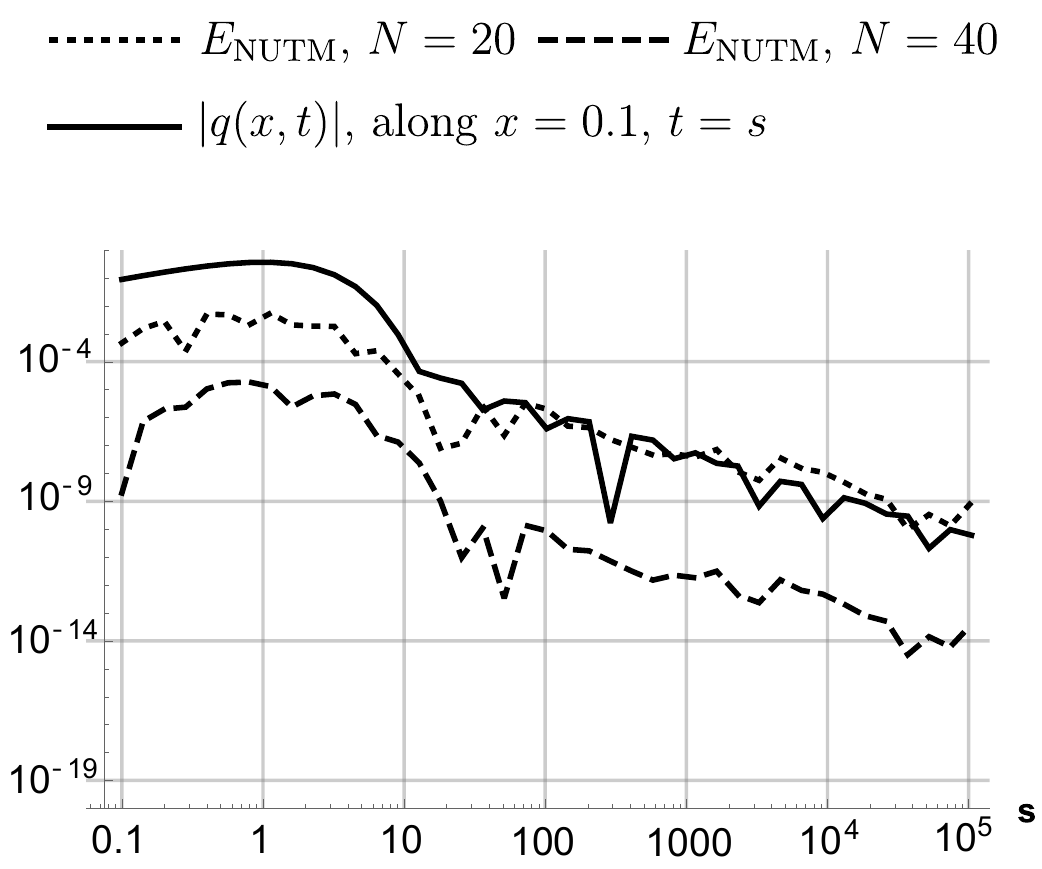}
\caption{}
\end{subfigure}

 \begin{subfigure}{0.3\textwidth} 
  \includegraphics[width=\textwidth]{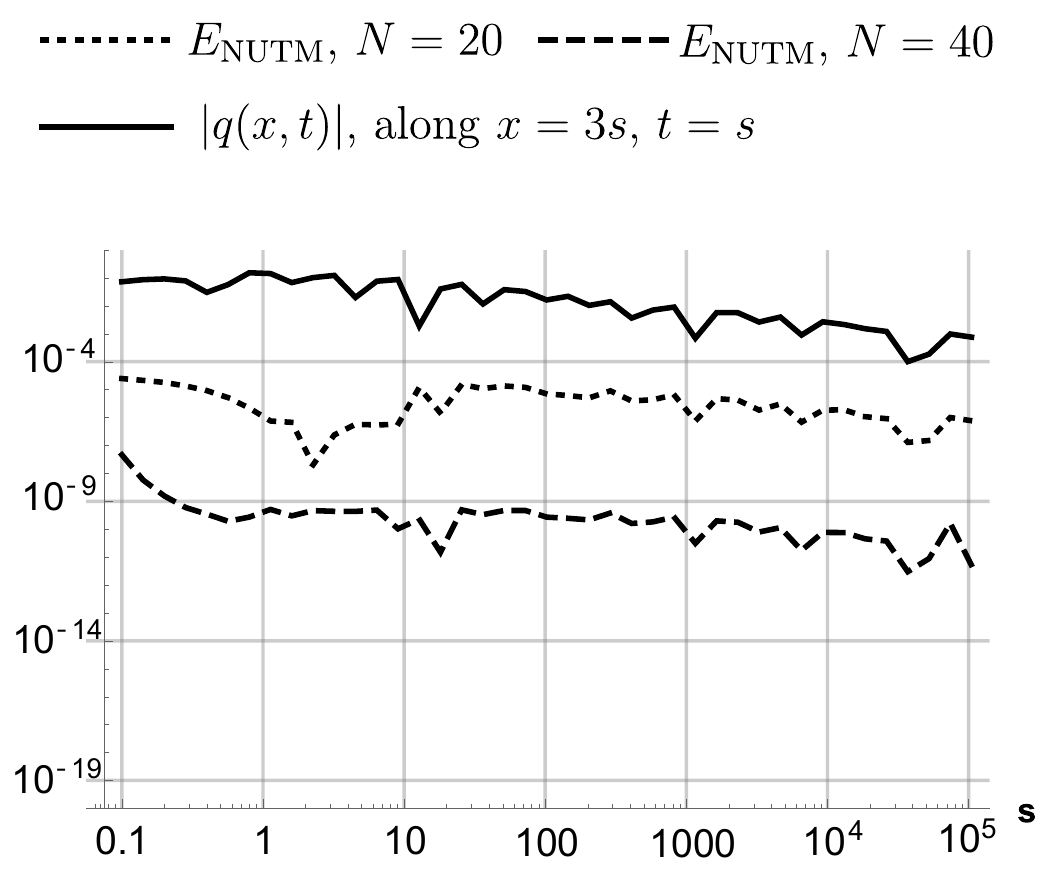}
\caption{}
\end{subfigure}

  }
  \caption{The absolute error $E_{\text{NUTM}}$ of the numerical solution to (\ref{lkdveq}) along (a) $x=s,t=0.1$,  (b) $x=0.1, t=s$, (c) $x=3s,t=s$ for $s\in [0.1,10^5]$. The computation using $N=20$ points for each segment in the contour (dotted) and using $N=40$ points for each segment in the contour (dashed) are plotted.}
  \label{lkdverr}
\end{figure}

\begin{figure}
  \makebox[\textwidth][c]{
  \includegraphics[width=0.9\textwidth]{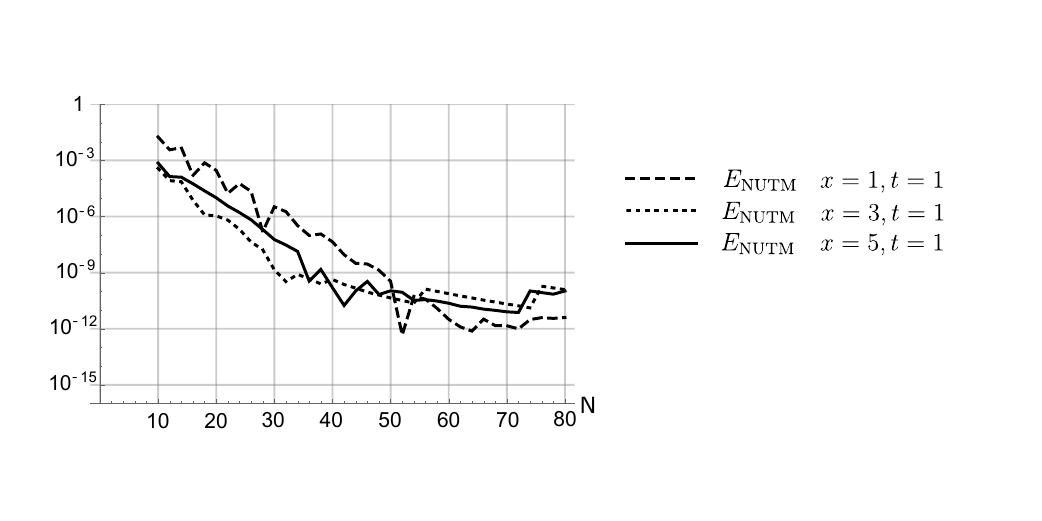}
  }
  \caption{The absolute error $E_{\text{NUTM}}$ against the number of collocation points $N$ per segment: computed with  $x=1, t=1$ (dashed), $x=3,t=1$ (dotted) and $x=5,t=1$ (solid).}
  \label{lkdverr2}
\end{figure}

\begin{figure}
  \makebox[\textwidth][c]{
 \begin{subfigure}{0.55\textwidth} 
  \includegraphics[width=1.\textwidth]{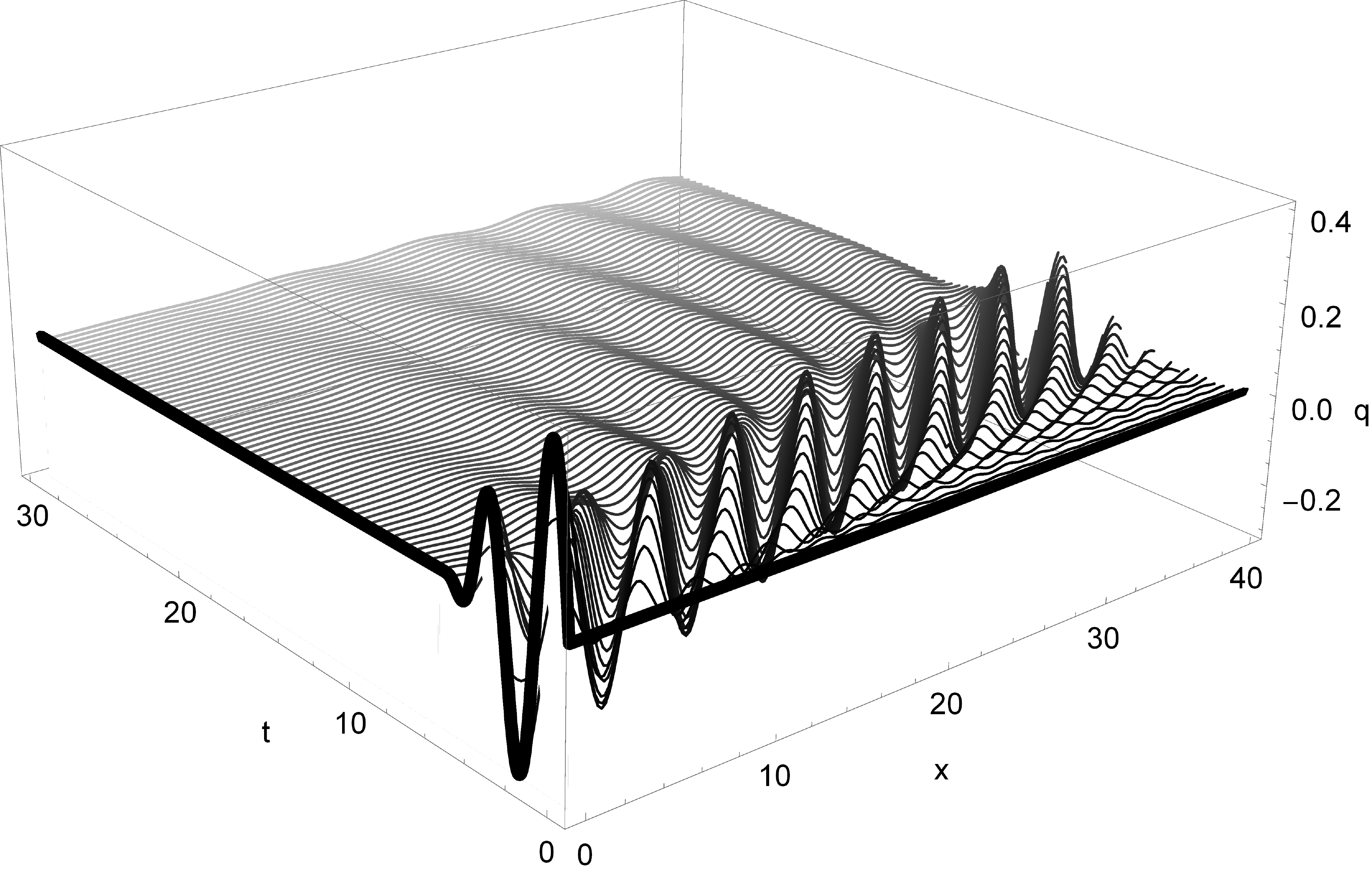}
\caption{}
\end{subfigure}

 \begin{subfigure}{0.41\textwidth} 
    \includegraphics[width=1.\textwidth]{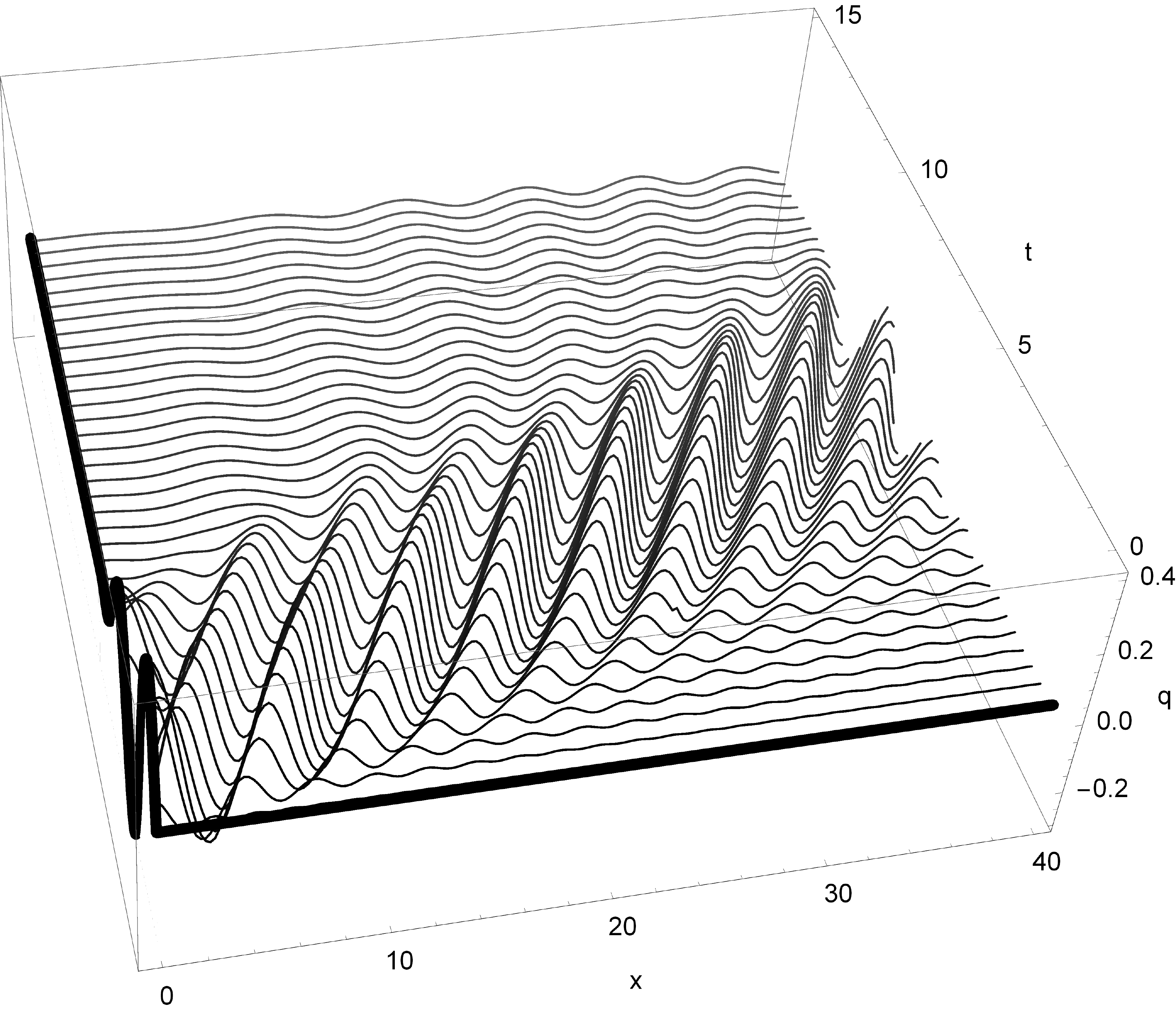}
\caption{}
\end{subfigure}

  }
  \caption{The numerical solution of (\ref{lkdveq}) with $g_0=\sin(2t) \phi(t/(2\pi))$ with $\phi(t)$ defined in (\ref{molifier}). The bold curves are the initial and Dirichlet boundary conditions. For small $t$, dispersive waves emanate from the boundary but the waves start to turn back because of advection as $t$ grows. Panel (a) shows the solution for $x\in [0,40],~t\in [0,30]$. Panel (b) shows the solution from a different angle in a shorter time interval $t\in[0,15]$.}
  \label{lkdvsol2}
\end{figure}

\section*{Acknowledgments}
The authors gratefully acknowledge support from the US National Science Foundation under grants NSF-DMS-1522677 (BD,XY), NSF-DMS-1753185 (TT) and NSF-DMS-1945652 (TT). Any opinions, findings, and conclusions or recommendations expressed in this material are those of the authors and do not necessarily reflect the views of the funding sources.

\section*{Appendix: The proof of uniform convergence of the NUTM applied to the heat equation.}
\label{sec_proof}
In this appendix, we prove the uniform convergence for Clenshaw-Curtis quadrature applied to the contour integrals for the heat equation in Section \ref{sec_heatdeform}. We use the following result to estimate the error of Clenshaw-Curtis quadrature. The constant $K$ for the integrals $I_1,~I_2$ and $B_0$ is given in Theorem \ref{thmi1} and Theorem~\ref{thmi3}.
\begin{theorem}[See \cite{trogdon}, for example.]
Let $u(k;x,t)$ be so that for $m=0,1,\ldots,M$, $\partial_k ^m u(k;x,t)$ are absolutely continuous for fixed $x,t$ and satisfy $\sup_{k\in [-1,1]} \abs{\partial_k^{M+1} u(k;x,t)}\leq K$ for all $x,t$. Define $i(u(\cdot;x,t))=\int u(k;x,t)dk$ and $i_n(u(\cdot;x,t))$ to be the approximation of $i(u(\cdot;x,t))$ obtained with Clenshaw-Curtis quadrature. Then $i_n(u(\cdot;x,t))$ converges to $i(u(\cdot;x,t))$ uniformly in $x,t$. More precisely, there exists $N>0$ such that for $n>N$,
\[
\sup_{x,t} \abs{i(u(\cdot;x,t))-i_n(u(\cdot;x,t))}\leq \frac{32K}{15M(2n+1-M)^M}.
\]
\end{theorem}
In Theorem \ref{thmi1} and Theorem \ref{thmi3}, we estimate the upper bound $K$ for each part of the integral in (\ref{heatsolintg}). 
The uniform convergence is considered in the domain bounded away from the $t=0$ and $x=0$. For $c>0$, we  define the region,  
\[
\Omega_{c}=\{(x,t): ~ x\geq c, ~ t\geq c \}.
\]
 
\begin{theorem}[Uniform convergence of $I_1$ and $I_2$ in (\ref{heatsolintg}) for the heat equation]
For any $\delta,\epsilon,c>0$, assume $q_0\in C_\delta^{\infty}$ and let $I_1^{\epsilon}$ be the truncation of the integral \footnote{The truncation depends on the prescribed tolerance $\epsilon$. As $\hat{q}_0$ is bounded on the contour, we can use the exponential to get a good choice for the trunction. See the proof for how the truncation is done.} 
\begin{align}
I_1 = \frac{1}{2\pi}\int_{\mathcal{C}^I_1} e^{ikx-\omega(k)t}\hat{q}_0(k)dk,
\label{i1heat}
\end{align}
such that  
\[
\sup_{(x,t)\in \Omega_{c}} \abs{I_1-I_1^{\epsilon}}< C_1(q_0,\delta,c) \epsilon,~~~C_1(q_0,\delta,c)>0.
\]
Then Clenshaw-Curtis quadrature applied to $I_1^{\epsilon}$ converges uniformly on $\Omega_{c}$. 
Hence $\mathcal{C}^I_1=\{a+ ih:a\in \mathbb{R}\}$ and $h=\min(x/2t,\delta)$ is  as defined in Section \ref{sec_heatj1}. Similarly, with the same assumptions, let $I_2^{\epsilon}$ be the truncation of the integral 
\begin{align}
I_2 = -\frac{1}{2\pi}\int_{\mathcal{C}^I_2} e^{ikx-k^2 t}\hat{q}_0(-k)dk,
\label{i2heat}
\end{align}
such that  
\[
\sup_{(x,t)\in \Omega_{c}} \abs{I_2-I_2^{\epsilon}}< C_2(q_0,\delta,c) \epsilon,~~~C_2(q_0,\delta,c)>0.
\]
Then Clenshaw-Curtis quadrature applied to $I_2^{\epsilon}$ converges uniformly on $\Omega_{c}$. 
Hence $\mathcal{C}^I_2=\{a+ ix/2t:a\in \mathbb{R}\}$.
\label{thmi1}
\end{theorem}

\begin{theorem}[Uniform convergence of $B_0$ in (\ref{heatsolintg}) for the heat equation]
For any $\gamma,\epsilon,c>0$, assume $g_0\in C_\gamma^{\infty}$ and let $B_0^{\epsilon}$ be the  truncation of the integral \footnote{As with Theorem 2, the truncation procedure is described in the proof.}
\begin{align}
B_0 = \frac{1}{\pi}\int_{\mathcal{C}^B_{0,a}} e^{ikx-k^2 t} 2ik\tilde{g}_0(k^2,t) dk+\frac{1}{2\pi}\int_{\mathcal{C}^B_{0,b}+\mathcal{C}^B_{0.c}} e^{ikx-k^2 t} 2ik\tilde{g}_0(k^2,t) dk,
\label{i3heat}
\end{align}
such that  
\[
\sup_{(x,t)\in \Omega_{c}} \abs{B_0-B_0^{\epsilon}}<C(g_0,\gamma,c)\epsilon,~~~C(g_0,\gamma,c)>0.
\]
Then Clenshaw-Curtis quadrature applied to $B_0^{\epsilon}$ converges uniformly on $\Omega_{c}$. The contour is defined in Section \ref{sec_heati3} where $\mathcal{C}^B_{0,a}= \{La+ix/2t:a\in [0,1], e^{-L^2t}=\epsilon \}$ is the horizontal segment of the contour and  $\mathcal{C}^B_{0,b}= \{L+ix/2t+L_2 e^{i\pi/4}a: a\in [0,\infty), e^{-L_2 x}=\epsilon \}$, $\mathcal{C}^B_{0,c}= \{-L+ix/2t+L_2 e^{-i\pi/4}a: a\in (-\infty,0], e^{-L_2 x}=\epsilon \}$ are the oblique segments of the contour with given tolerance $\epsilon>0$.  
\label{thmi3}
\\
\end{theorem}


\begin{proof}[Proof of Theorem \ref{thmi1}]
For given tolerance $\epsilon>0$, $I_1$ is truncated to $I_1^{\epsilon}$ of length $2L$ with  $e^{-L^2t}=\epsilon$. We introduce the change of variables $k=La+ih$. The integral with $a>1$ is cut off,
\begin{align*}
I_1 = &\frac{Le^{-hx}}{2\pi}\int_{-\infty}^{\infty} e^{iLax-(La+ih)^2 t}\hat{q}_0(La+ih)da\\
    = &\frac{Le^{-hx}}{2\pi}\int_{-1}^{1} e^{iLax-(La+ih)^2 t}\hat{q}_0(La+ih)da
    + \frac{Le^{-hx}}{2\pi}\int_{\abs{a}>1} e^{iLax-(La+ih)^2 t}\hat{q}_0(La+ih)da\\
    =& I_1^{\epsilon}+\frac{Le^{-hx}}{2\pi}\int_{\abs{a}>1} e^{iLax-(La+ih)^2 t}\hat{q}_0(La+ih)da.
\end{align*}
The second integral is dropped and the induced truncation error is bounded by
\begin{align*}
\abs{\frac{Le^{-hx}}{2\pi}\int_{\abs{a}>1} e^{iLax-(La+ih)^2 t}\hat{q}_0(La+ih)da} 
\leq & \frac{Le^{-hx}}{2\pi}\int_{\abs{a}>1} \abs{ e^{iLax-(La+ih)^2 t}\hat{q}_0(La+ih)}da \\
\leq & \frac{Le^{-hx}}{2\pi}\int_{\abs{a}>1} e^{-L^2ta^2+h^2t}\abs{\hat{q}_0(La+ih)}da \\
\leq & \norm{\hat{q}_0(\cdot+ih)}_{\infty}\frac{Le^{-h(x-ht)}}{2\pi} \epsilon \int_{\abs{a}>1} a e^{-L^2t(a^2-1)}da \\
\leq & \norm{\hat{q}_0(\cdot+ih)}_{\infty}\frac{Le^{-h(x-ht)}}{2\pi} \epsilon \int_{0}^{\infty}  e^{-L^2t s}ds  \\
\leq & \norm{\hat{q}_0(\cdot+ih)}_{\infty}\frac{Le^{-h^2t}}{2\pi}  \frac{\epsilon}{(-\ln \epsilon)}.
\end{align*}
Since $t$ is bounded from below, $L$ is bounded from above. The truncation error is therefore $\mathcal{O}(\epsilon)$, uniformly in $(x,t) \in \Omega_c$. 

Uniform convergence to $I_1^{\epsilon}$ requires the derivative of the integrand in $I_1^{\epsilon}$ to satisfy 
\[
\sup_{a\in [-1,1]}  \frac{Le^{-hx}}{2\pi}\abs{ \partial^2_a \left( e^{iLax-(La+ih)^2 t}\hat{q}_0(La+ih) \right) }\leq M,
\]
for all $x,t$. Notice that the derivatives of the exponential only introduce polynomial terms and $\hat{q}_0(k)$ is bounded and analytic in $\{k:\im{k}\leq \delta\}$ which implies that $\partial_k \hat{q}_0(k)$ and $\partial^2_k\hat{q}_0(k)$ are bounded on the contour. It suffices to show
\begin{align*}
\sup_{a\in [-1,1]}  \abs{ \partial^2_a \left( \frac{Le^{-hx}}{2\pi}
\cdot e^{iLax-(La+ih)^2 t}\hat{q}_0(La+ih) \right) } &\leq \sup_{a\in [-1,1]} \abs{   \frac{Le^{-hx}}{2\pi} \cdot e^{iLax-(La+ih)^2 t} P(a,Lx,L^2 t,Lht)} \\
&= \sup_{a\in [-1,1]}    \frac{Le^{-hx}}{2\pi}
\cdot e^{-L^2t a^2+h^2t }\abs{P(a,Lx,L^2 t,Lht)},
\end{align*}
where $P$ is a polynomial with positive coefficients.

When $h=x/2t\leq \delta$,
\begin{align*}
\sup_{a\in [-1,1]}    \frac{Le^{-hx}}{2\pi}
\cdot e^{-L^2t a^2+h^2t }\abs{P(a,Lx,L^2 t,Lht)}
  =&\sup_{a\in [-1,1]}  \frac{Le^{-x^2/4t-L^2t a^2}}{2\pi}  
  \abs{P(a,Lx,L^2 t,Lx/2) }\\
  \leq  & \frac{e^{-x^2/4t}}{2\pi} Q(x^2/4t)
\leq M_1 < \infty,
\end{align*}
where $Q$ is a polynomial with positive coefficients and we have used that $L^2t$ is constant.

When $h=\delta< x/2t $, 
\begin{align*}
\sup_{a\in [-1,1]}   \frac{Le^{-hx}}{2\pi}
\cdot e^{-L^2t a^2+h^2t } \abs{P(a,Lx,L^2 t,Lht)}  \leq&
\sup_{a\in [-1,1]}   e^{-\delta x/2-L^2t a^2 } \abs{P(a,Lx,L^2 t,Lx/2)} \\
\leq& \sup_{a\in [-1,1]}   e^{-\delta x/2 } Q_2(x) \leq M_2 < \infty,
\end{align*}
where $Q_2$ is a polynomial with positive coeffcients. As a result, the second derivative of the integrand of~(\ref{i1heat}) is uniformly bounded by $M=\max(M_1,M_2)$ independent of $x,t$. Together with the smoothness of the integrand, uniform convergence is obtained using Theorem 1.
We skip the calculation for $I_2$ as it follows the calculation for $I_1$.
\end{proof}

\begin{proof}[Proof of Theorem \ref{thmi3}]
First, we prove the uniform convergence for the integral along $\mathcal{C}^B_{0,a}$. Introduce the
change of variables $k=La+ix/2t$. 
\begin{align*}
B_0|_{\mathcal{C}^B_{0,a}} = & \frac{1}{2\pi}\int_{\mathcal{C}^B_{0,a}}  e^{ikx-k^2 t}2k\tilde{g}_0(k^2,t) dk\\
    = &
    \frac{Le^{-x^2/2t}}{2\pi}\int_{-1}^{1} e^{iLax-(La+ix/2t)^2 t} 2(La+ix/2t)\tilde{g}_0((La+ix/2t)^2,t) da\\
    = &    \frac{Le^{-x^2/4t}}{\pi}\int_{-1}^{1} e^{-L^2t a^2}(La+ix/2t)\int_{0}^{t}e^{(La+ix/2t)^2s}g(s)ds da.
\end{align*}
Using Theorem 1, uniform convergence requires the boundedness of the second derivative of the integrand
\[
B_a= \sup_{a\in [0,1]}  \frac{Le^{-x^2/4t}}{\pi}\abs{\partial^2_a  e^{-L^2t a^2}(La+ix/2t)\int_{0}^{t}e^{(La+ix/2t)^2s-\gamma s}g(s)e^{\gamma s}ds  }\leq M,
\]
for all $(x,t) \in \Omega_c$. Since $\norm{g e^{\gamma (\cdot)}}_{\infty}<\infty$, after a lengthy computation, 
\begin{align*}
B_a\leq \sup_{a\in [0,1]}&    \frac{\norm{g e^{\gamma (\cdot)} }_{\infty}}{\abs{-4 a^2 t^2 + 4 \gamma t^2 - 4 i a t x + x^2}^3}\left(e^{-a^2t-x^2/4t} P_1+e^{-\gamma t-x^2/2t} P_2\right) ,
\end{align*}
where $P_1,P_2$ are polynomials in $x,t$ and $a$, with positive coefficients, and the growth for large $x,t$ is controlled by the exponential and the denominator in front of $P_1,P_2$. As a result $B_a\leq M$ and the integral on $\mathcal{C}^B_{0,a}$ is computed with uniform accuracy. 
Lastly, we show the uniform convergence for the integral along the oblique segment $\mathcal{C}^B_{0,b}$. The proof for the integral along $\mathcal{C}^B_{0,c}$ follows directly by symmetry. 
We introduce the change of variables $k=L+ix/2t+L_2 (1+i) a$. The integral with $a>1$ is separated,
\begin{align*}
B_0|_{\mathcal{C}^B_{0,b}} =  & \frac{L_2 (1+i)}{\pi} \left( \int_{0}^{1} + \int_{1}^{\infty} \right)  e^{ikx-k^2 t}k\tilde{g}_0(k^2,t)\Big|_{k=L+ix/2t+L_2 (1+i) a}da\\
    = &
    B_0^{\epsilon}|_{\mathcal{C}^B_{0,b}}+  \frac{L_2 (1+i)}{\pi} \int_{1}^{\infty}  e^{ikx-k^2 t}k\tilde{g}_0(k^2,t)\Big|_{k=L+ix/2t+L_2 (1+i) a}da.
\end{align*}
The second integral is dropped and the induced truncation error is bounded by
\begin{align*}
\abs{B_0|_{\mathcal{C}^B_{0,b}} -B_0^{\epsilon}|_{\mathcal{C}^B_{0,b}}} 
\leq & \frac{L_2 e^{-L^2t-x^2/(4t)}}{\pi} \int_{1}^{\infty} \abs{ e^{-2LL_2 ta}(Lt+L_2 ta+\frac{x}{2t}) \tilde{g}_0((L+ix/2t+L_2 (1+i) a)^2,t) }da \\ 
\leq & \frac{L_2 e^{-x^2/(4t)}\norm{g_0 e^{\gamma (\cdot)} }_{\infty} }{\pi} \int_{1}^{\infty} \abs{ \left(Lt+L_2 ta+\frac{x}{2t}\right) \frac{\left(e^{-  L_2 xa-x^2/(4 t)-\gamma  t }-e^{-L^2 t-2LL_2 ta}\right)}{( L -x/(2t)) (2 a L_2 + L +x/(2t))- \gamma} }da\\
\leq & e^{-x^2/(4 t)}\norm{g_0 e^{\gamma (\cdot)} }_{\infty}  \left(e^{-L_2 x -\gamma t}  P_3 + e^{-L^2 t-2LL_2 t} P_4\right),
\end{align*}
where $P_3,P_4$ are polynomials of $x,t$ with positive coefficients. Since the decaying exponentials dominate the growth of the polynomial, the truncation error is $\mathcal{O}(\epsilon)$, uniformly in $(x,t) \in \Omega_c$ with $e^{-L_2x}=\epsilon$ and $e^{-L^2t}=\epsilon$. 
Using Theorem 1, uniform convergence requires the boundedness of the second derivative of the integrand
\[
B_b=\sup_{a\in [0,1]}  \frac{L_2e^{i \pi/4}}{2\pi} \abs{ \partial^2_a \left(e^{ikx-k^2 t}2k\tilde{g}_0(k^2,t)\Big|_{k=L+ix/2t+L_2 e^{i\pi/4} a}   \right) }\leq M_b,
\]
for all $x,t$. After computing the derivatives, 
\begin{align*}
B_b \leq \sup_{a\in [0,1]}& \frac{\norm{g_0 e^{\gamma (\cdot)} }_{\infty}}{\abs{4 \gamma t^2 - (2 L t + (2 + 2 i) a L_2 t + i x)^2}^3} \left(e^{-L^2 t - 2 a L L_2 t - x^2/(4 t)} P_5 + e^{-\gamma t - a L_2 x - x^2/(2 t)} P_6 \right) ,
\end{align*}
where $P_5,P_6$ are polynomials of $x,t,a$ with positive coefficients. The poles are removable since the integrand is analytic in $k$. In this case, the exponentials dominate the growth of the polynomial. Hence, $B_b\leq M$. The second derivative of the integrand of (\ref{i1heat}) is uniformly bounded by $M$, independent of $x,t$. Together with the smoothness of the integrand, uniform convergence is obtained using Theorem 1.

\end{proof}


\begin{thebibliography}{1}

\bibitem{barros2019} F. P. J. de Barros, M. J. Colbrook and A. S. Fokas. A hybrid analytical-numerical method for solving advection-dispersion problems on a half-line. International Journal of Heat and Mass Transfer 139 (2019), 482-491.

\bibitem{deconinck2014} B. Deconinck, T. Trogdon and V. Vasan. The method of Fokas for solving linear partial differential equations. SIAM Review 56 1 (2014), 159-186.

\bibitem{flyer2008} N. Flyer and A. S. Fokas. A hybrid analytical-numerical method for solving evolution partial diffrerential equations. I. The half-line. Proceedings of the Royal Society A 464 (2008), 1823-1849.

\bibitem{fokas1997} A. S. Fokas. A unified transform method for solving linear and certain nonlinear PDEs. Proceedings of the Royal Society A 453 (1997), 1411-1443.

\bibitem{fokas2002a} A. S. Fokas. A new transform method for evolution partial differential equations. IMA Journal of Applied Mathematics 67 (2002), 559-590.

\bibitem{fokas2002b} A. S. Fokas. Integrable nonlinear evolution equations on the half-line. Communications in Mathematical Physics 230 (2002), 1-39.

\bibitem{fokas2005} A. S. Fokas. The nonlinear Schr\"odinger equation on the half-line. Nonlinearity 18 (2005), 1771-1882.

\bibitem{utmbook} A. S. Fokas. A Unified Approach to Boundary Value Problems. SIAM, Philadelphia, PA 2008.

\bibitem{fokas2009} A. S. Fokas, N. Flyer, S. A. Smitheman and E.A. Spence. A semi-analytical numerical method for solving evolution and elliptic partial differential equations. Journal of Computational and Applied Mathematics 227 (2009), 59-74.

\bibitem{gibbs2019} A. Gibbs, D. Hewett, D. Huybrechs and E. Parolin. Fast hybrid numerical-asymptotic boundary element methods for high frequency screen and aperture problems based on least-squares collocation. arXiv:1912.09916 [math.NA]

\bibitem{huybrechs} D. Huybrechs and A. Gibbs. PathFinder: a toolbox for oscillatory integrals by deforming into the complex plane. \verb|https://github.com/AndrewGibbs/PathFinder|.

\bibitem{iserles2006} A. Iserles, S. P. Norsett and S. Olver. Highly oscillatory quadrature: The story so far. Numerical Mathematics and Advanced Applications (2006).

\bibitem{kesici2018} E. Kesici, B. Pelloni, T. Pryer and D. Smith. A numerical implementation of the unified Fokas transform for evolution problems on a finite interval. European Journal of Applied Mathematics 29 3 (2018), 543-567.

\bibitem{millerbook} P. D. Miller. Applied asymptotic analysis. AMS, Providence, RI 2006.

\bibitem{olver2012} S. Olver and A. Townsend. A fast and well-conditioned spectral method. SIAM Review 55 (2012), 462-489.

\bibitem{papa2009} T. S. Papatheodorou and A. N. Kandili. Novel numerical techniques based on Fokas transforms, for the solution of initial boundary value problems. Journal of computational and applied mathematics 227 (2009), 75-82.

\bibitem{trefethen2008} L. N. Trefethen. Is Gauss quadrature better than Clenshaw-Curtis? SIAM Review 50 (2008), 67-87.

\bibitem{trogdon} T. Trogdon. A unified numerical approach for the nonlinear Schr\"odinger equations. In A.S. Fokas and B. Pelloni, editors, Unified Transform method for boundary value problems: applications and advances, chapter 8, 259-292, SIAM, Philadelphia, PA 2015.

\bibitem{trogdon2012a} T. Trogdon and S. Olver. Numerical inverse scattering for the focusing and defocusing nonlinear Schr\"{o}dinger equations, Proceedings of the Royal Society of London A 469 (2013).

\bibitem{trogdon2012b} T. Trogdon, S. Olver and B. Deconinck. Numerical inverse scattering for the Korteweg-de Vries and modified Korteweg-de Vries equations, Physica D 241 11 (2012), 1003-1025.

\bibitem{trogdon2019} T. Trogdon and G. Biondini. Evolution partial differential equations with discontinuous data, Quarterly of Applied Mathematics 77 (2019), 689-726.

\bibitem{uspensky1928} J. V. Uspensky. On the convergence of quadrature formulas related to an infinite interval, Transactions of the American Mathematical Society 30  (1928), 542-59.

\bibitem{yang2019}  B. Deconinck, T. Trogdon and X. Yang. Numerical inverse scattering for the sine-Gordon equation, Physica D. 399 (2019), 159-172.


\end{thebibliography}
\end{document}